\documentclass[11pt,reqno]{amsart}

\usepackage[utf8]{inputenc}
\usepackage{mathtools}
\usepackage{amsmath}
\usepackage{amsfonts}
\usepackage{mathrsfs}
\usepackage{bbm}
\usepackage{fancyhdr,tensor}
\usepackage{amssymb}
\usepackage{amsthm, comment}
\usepackage[margin=1.0in]{geometry}

\usepackage[colorlinks=true, pdfstartview=FitV, linkcolor=blue, citecolor=blue, urlcolor=blue]{hyperref}

\usepackage{xcolor}
\usepackage{ytableau}
\usepackage{tikz}
\usetikzlibrary{cd}
\usepackage{tikz-cd}
\usepackage{hyperref}
\usepackage{makecell}
\usepackage{cleveref}
\usepackage{bm}
\ytableausetup{mathmode, boxframe=normal, boxsize=2em}

\definecolor{darkblue}{rgb}{0.0,0,0.7} % darkblue color
\definecolor{darkred}{rgb}{0.7,0,0} % darkred color
\definecolor{darkgreen}{rgb}{0, .6, 0} % darkgreen color

\usepackage[colorinlistoftodos]{todonotes}

\newtheorem{theorem}{Theorem}[section]
\newtheorem{lemma}[theorem]{Lemma}
\newtheorem{corollary}[theorem]{Corollary}

\newtheorem{proposition}[theorem]{Proposition}

\theoremstyle{definition}
\newtheorem{definition}[theorem]{Definition}
\newtheorem{example}[theorem]{Example}
\newtheorem{remark}[theorem]{Remark}

\numberwithin{equation}{section}

\newcommand{\A}[2]{\tensor[_{\; #1}]{\hspace{-1mm}\mathcal{T}}{_{#2}}}
\newcommand{\R}[2]{\tensor[_{#1}]{\mathcal{P}}{_{#2}}}

\newcommand{\G}[2]{\tensor[_{#1}]{{G}}{_{#2}}}
\newcommand{\F}[2]{\tensor[_{\; #1}]{{F}}{_{#2}}}
\newcommand{\E}[2]{\tensor[_{\, #1}]{{E}}{_{#2}}}

\newcommand{\Q}[2]{\tensor[_{#1}]{{Q}}{_{#2}}}
\newcommand{\U}[2]{\tensor[_{#1}]{{U}}{_{#2}}}
\newcommand{\W}[2]{\tensor[_{#1}]{{W}}{_{#2}}}
\newcommand{\HH}[2]{\tensor[_{#1}]{{H}}{_{#2}}}
\newcommand{\II}[2]{\tensor[_{#1}]{{I}}{_{#2}}}
\newcommand{\ZZ}[2]{\tensor[_{#1}]{{Z}}{_{#2}}}
\newcommand{\BB}[2]{\tensor[_{#1}]{{B}}{_{#2}}}

\newcommand{\fF}[2]{\tensor[_{#1}]{\mathcal{F}}{_{#2}}}
\newcommand{\fE}[2]{{}_{#1}{\mathcal{E}}{_{#2}}}

\newcommand{\B}{\mathscr{B}}

\newcommand{\cE}{\mathsf{E}}
\newcommand{\cF}{\mathsf{F}}

\newcommand{\fk}{\mathbbm{k}}

\newcommand{\g}{\mathfrak{g}}
\newcommand{\SSYT}{\operatorname{SSYT}}

\newcommand{\COL}{\operatorname{Col}}
\newcommand{\GT}{\mathrm{GT}}
\newcommand{\GGTT}{\mathcal{GT}}
\newcommand{\wt}{\operatorname{wt}}
\newcommand{\height}{\operatorname{ht}}
\newcommand{\newword}[1]{\emph{\textbf{#1}}}
\newcommand{\Span}{\operatorname{Span}}
\newcommand{\End}{\operatorname{End}}
\newcommand{\rep}{\operatorname{rep}}

\newcommand{\maxSpec}{\operatorname{maxSpec}}
\newcommand{\ev}{\operatorname{ev}}
\newcommand{\Ev}{\operatorname{\mathsf{Ev}}}
\newcommand{\Quo}{\operatorname{\mathsf{Pr}}}

\newcommand{\hideqed}{\renewcommand{\qed}{}}

\newlength\cellsize 

\savebox6{%
\begin{picture}(20,20)
\put(0,0){\line(1,0){20}}
\put(0,0){\line(0,1){20}}
\put(20,0){\line(0,1){20}}
\put(0,20){\line(1,0){20}}
\end{picture}}

\newcommand\cellify[1]{\def\thearg{#1}\def\nothing{}%
\ifx\thearg\nothing\vrule width0pt height\cellsize depth0pt%
  \else\hbox to 0pt{\usebox6\hss}\fi%
  \vbox to 20\unitlength{\vss\hbox to 20\unitlength{\hss$#1$\hss}\vss}}

\newcommand\tableau[1]{\vtop{\let\\=\cr
\setlength\baselineskip{-10000pt}
\setlength\lineskiplimit{10000pt}
\setlength\lineskip{0pt}
\halign{&\cellify{##}\cr#1\crcr}}}

\input xy
\usepackage[all]{xy}
\xyoption{line}
\xyoption{arrow}
\xyoption{color}
\SelectTips{cm}{}
\usepackage{tikz}
\usetikzlibrary{decorations.markings}
\usetikzlibrary{decorations.pathreplacing}
\newcommand{\hackcenter}[1]{
 \xy (0,0)*{#1}; \endxy}

\tikzstyle directed=[postaction={decorate,decoration={markings,
    mark=at position #1 with {\arrow{>}}}}]
\tikzstyle rdirected=[postaction={decorate,decoration={markings,
    mark=at position #1 with {\arrow{<}}}}]

\usetikzlibrary{calc}
\usepackage{relsize}

\tikzset{fontscale/.style = {font=\relsize{#1}}
    }

\author[S. Daugherty]{Spencer Daugherty}
\address[S. Daugherty]{
    Department of Mathematics, 
    University of Colorado Boulder, 
    CO, USA 
}
\email{spencer.daugherty@colorado.edu}

\author[N. Gonz{\'a}lez]{Nicolle Gonz{\'a}lez}
\address[N. Gonz{\'a}lez]{
    Department of Mathematics, 
    University of British Columbia, 
    Vancouver, BC, Canada 
}
\email{nicolle@math.ubc.ca}

\author[B. Muniz]{B\'arbara Muniz}
\address[B. Muniz]{
    Institute of Mathematics, 
    Jagiellonian University in Krak{\'o}w, 
    Poland 
}
\email{barbara.santanamuniz@doctoral.uj.edu.pl}

\author[P. S. Ocal]{Pablo S. Ocal}
\address[P. S. Ocal]{
    Department of Mathematics, 
    University of British Columbia, 
    Vancouver, BC, Canada 
}
\email{socal@math.ubc.ca}

\author[J. Pan]{Jianping Pan}
\address[J. Pan]{
    School of Mathematical and Statistical Sciences, 
    Arizona State University, 
    AZ, USA 
}
\email{jianping.pan@asu.edu}

\author[J. Torres]{Jacinta Torres}
\address[J. Torres]{
    Institute of Mathematics, 
    Jagiellonian University in Krak{\'o}w, 
    Poland 
}
\email{jacinta.torres@uj.edu.pl}

\title{Structure and Geometry of the Tableaux Algebra}



\thanks{For the purpose of open access, the authors have applied a CC-BY copyright 
license to any Author Accepted Manuscript (AAM) version arising from this paper.}

\begin{document}

\begin{abstract}
    We study the monoid algebra 
    $\tensor[_{\; n}]{\hspace{-1mm}\mathcal{T}}{_{m}}$ 
    of semistandard Young tableaux, which coincides with the Gelfand--Tsetlin 
    semigroup ring $\mathcal{GT}_{n}$ when $m = n$. 
    Among others, we show that this algebra is 
    commutative, Noetherian, reduced, Koszul, and Cohen--Macaulay. 
    We provide a complete 
    classification of its maximal ideals and compute the topology of its 
    maximal spectrum. Furthermore, we classify its irreducible modules and 
    provide a faithful semisimple representation. 
    We also establish that its associated variety coincides with 
    a toric degeneration of certain partial flag varieties constructed by 
    Gonciulea--Lakshmibai. As an application, we show that this algebra yields 
    injective embeddings of $\mathfrak{sl}_n$-crystals, 
    extending a result of Bossinger--Torres. 
\end{abstract}

\maketitle 

%\tableofcontents

%%%%%%%%%%%%%%%%%%%%%%%%%%%%%%%%%%%%%%%%%%%%%%%%%%%%%%%%%%%%%%%%%%%%%%%%%%%%%%%%

\section{Introduction}

Semistandard Young tableaux are ubiquitous in combinatorics and representation theory, 
and have seen applications in a wide range of fields, from probability theory to 
algebraic geometry and beyond. Consequently, many algebraic structures on the set 
of semistandard Young tableaux have been widely studied, arguably the most famous 
of which is the plactic monoid; see for example 
\cite{Lothaire, LS-plactic, knuth, superplactic, KimProtsak, poirier1995algebres}. 
In this paper we study a monoid structure on this set that does \textit{not} arise 
from an insertion algorithm, but rather from the naive concatenation of rows. 
We prove that the algebra associated to this monoid satisfies some remarkable algebraic, 
combinatorial, and geometric properties. 

Given semistandard Young tableaux $T$ and $T'$, each with at most $n$ rows 
and entries in $\{1,\dots,m\}$, we declare $T \star T'$ to be 
the semistandard Young tableau obtained by horizontally concatenating the rows 
and then sorting the entries within each row. 
This binary operation endows the set $\SSYT_m^n$ 
of semistandard Young tableaux with at most $n\leq m$ rows and entries in 
$\{1,\dots,m\}$ with the structure of a commutative monoid. 
The \newword{tableaux algebra} $\A{n}{m}$ is the monoid algebra of $\SSYT_m^n$, 
taken over $\fk$, an algebraically closed field of characteristic zero. 
When $m = n$, notorious incarnations of the monoid $\SSYT_m^n$ arise in connections 
with cluster algebras and canonical bases, Gelfand--Tsetlin polytopes, 
and toric degenerations of flag varieties; see, for example 
\cite{chang-duan-fraser-li:quantum, li:canonical, alexandersson:gt, KimProtsak, KM03, HM11}. 
However, a study of the fundamental algebraic structure and spectrum of $\A{n}{m}$ 
seems to be lacking in the literature. We remedy this here, establishing some of its 
fundamental structural properties, providing a complete description 
of its maximal spectrum, and exploring some further applications to 
algebraic geometry and representation theory. 

We begin by showing that $\A{n}{m}$ is a Noetherian, reduced, and Jacobson ring 
(Proposition \ref{prop:IntegralDomain-Noetherian-Reduced-Jacobson}). 
The tableaux algebra admits a natural presentation by column tableaux, 
which enables us to realize it as an explicit quotient of 
a finitely generated polynomial ring (Corollary \ref{cor:tensorprod}), 
and we denote its defining ideal by $\R{n}{m}$. 
In turn, by precisely identifying 
the generating relations of $\R{n}{m}$, 
we establish that not only is this algebra quadratic, it is in fact Koszul (Theorem \ref{thm:quadratic}, Corollary \ref{cor:koszul}). 
Moreover, we enumerate the minimal generating set for $\R{n}{m}$, 
which we later show forms a Gr{\"o}bner basis (Theorem \ref{thm:enumeration}, 
Corollary \ref{cor:enumeration}, Corollary \ref{cor:toricdegeneration}, 
Corollary \ref{cor:pluckerenumeration}). In particular, we provide 
an explicit combinatorial description and closed formulas for the 
\newword{Pl{\"u}cker--Grassmann relations} and the \newword{incidence relations}. 
To the best of the author's knowledge, this is the first time that such an 
enumeration of these relations appears in the literature. 

With the defining ideal in hand, we turn our attention to its associated variety 
$\mathcal{V}(\R{n}{m})$. We show that the tableaux algebra can be decomposed as a 
certain tensor product, with a number of free variables that depends on the difference 
between $m$ and $n$ (Corollary \ref{cor:tensorprod}). From this decomposition, 
it is evident that the tableaux algebra contains a polynomial subring in either 
two or three variables, which significantly complicates the study of its prime ideals. 
The prime spectrum is well understood for bivariate polynomial rings, 
but the analogous description for three or more variables is a difficult open problem 
in commutative algebra. The tableaux algebra is far more complicated than a polynomial ring, 
and although we are optimistic that some of the combinatorial techniques 
we present here are useful for classifying certain classes of prime ideals, 
a full description of its prime spectrum is currently out of reach. Notwithstanding this, 
we are able to classify the maximal ideals of $\A{n}{m}$ (Theorem \ref{thm:Maxideals}), 
and extensively study them via certain evaluation maps. This classification identifies each 
maximal ideal $M_{\underline{t}}$ with a tuple $\underline{t} \in \mathcal{V}(\R{n}{m}) \times \fk^r$, 
where $r$ is the number of aforementioned free variables. Whenever a tuple has all entries nonzero, 
we say that it corresponds to an \newword{ordinary} maximal ideal, 
which we study in more detail. 
For example, we show that the ordinary maximal ideals of $\A{n}{m}$ form 
a commutative group $\mathcal{M}_{ord}(n,m)$ under the binary operation 
$M_{\underline{t}} \circledast M_{\underline{t'}} \coloneqq M_{\underline{tt'}}$ 
(Proposition \ref{prop:maxidealgroup}). 

To understand the topology of $\maxSpec(\A{n}{m})$, we exploit 
an analogous characterization of the maximal ideals of the polynomial ring 
$\fk[X]$ in $nm$ variables, identifying each of its maximal ideals with an $n\times m$ matrix. 
This description yields a realization of $\mathcal{M}_{ord}(n,m)$ as a certain quotient 
of the group of ordinary maximal ideals of $\fk[X]$, which we again describe explicitly 
via evaluation homomorphisms (Theorem \ref{thm:EvalsMaxIdeal}). 
This correspondence enables a complete description of the topology of $\maxSpec(\A{n}{m})$ 
(Theorem \ref{thm:opens-max-spec}), which in turn gives rise to a continuous map 
from the maximal spectrum of $\fk[X]$ to that of $\A{n}{m}$ 
(Corollary \ref{cor:continuous-map-max-spec}). 

Understanding the maximal spectrum turns out to be particularly useful 
for a basic understanding of the representation theory of the tableaux algebra. 
We give a complete description of its finite dimensional irreducible modules 
(Corollary \ref{coro:finite-dimensional-irreducible}), and since these are 
in bijection with its maximal ideals, we immediately obtain that $\A{n}{m}$ 
has no infinite dimensional irreducible modules. 
In particular, we are in fact giving 
a complete classification of the simple modules of the tableaux algebra. 
Since $\A{n}{m}$ is semiprimitive (but not primitive), it has 
a faithful semisimple module, which must be infinite dimensional 
by the above classification of simple modules. This faithful semisimple module 
does not arise from the natural action on a polynomial ring, but rather we describe 
one in terms of its maximal ideals (Proposition \ref{prop:faithfulssmodule}), 
showing that the maximal spectrum fully determines the tableaux algebra. 
Unfortunately, structural classification results beyond these are essentially hopeless, 
as one may employ classical facts about the Drozd ring or the representation theory 
of free associative algebras to justify that a complete characterization 
of the representations of $\A{n}{m}$ is not a tractable pursuit. 

Although to the best of our knowledge the tableaux algebra 
has not been explicitly studied in the literature for $m$ and $n$ \textit{distinct}, 
when $m = n$ it is isomorphic to the well-known 
\newword{Gelfand--Tsetlin semigroup ring} $\GGTT_n$, as in \cite{GTpatterns}. 
This ring arises in the study of polytopes, quantum groups, cluster algebra, 
crystals, flag varieties, and many other areas of algebraic combinatorics; 
see for example 
\cite{KnutsonTao, li:canonical, KimProtsak, BL09, Seshadri, conescrystalspatterns, alexandersson:gt}. 
Crucially, $\GGTT_n$ is a flat degeneration of the ring of Pl{\"u}cker coordinates 
quotiented by the so-called Pl{\"u}cker relations; see \cite{GL96, KM03} 
(and Theorem \ref{thm:degeneration-n=m}). The resulting ring, termed the 
\newword{Pl{\"u}cker algebra}, is a fundamental object in combinatorial commutative algebra 
and geometry, as its relations encode the defining relations of the complete flag variety. 
Flat degenerations of the Pl{\"u}cker algebra, and the associated toric degenerations of 
flag and Schubert varieties, have received ample attention over the course of several decades, 
see \cite{Lakshmibai, mirrorsymmetry, bea_polyhedral, chiviri, caldero:toric, dehy, fgl}. 
The study carried out by Gonciulea--Lakshmibai in \cite{GL96}, 
where they describe a particular degeneration of the algebra of global sections 
of the space of partial flags, is particularly relevant for us. 
We realize $\A{n}{m}$ as the so-called Hibi algebra of the poset of column tableaux, 
see \cite{Hibi}, which allows us to show that for $m > n$ the tableaux algebra coincides 
with said flat degeneration of Gonciulea--Lakshmibai. 
Consequently, $\A{n}{m}$ is a flat degeneration of the algebra of global sections 
$\bigoplus_{\underline{a}}\Gamma(SL_m/\Q{n}{m},L^{\underline{a}})$ 
(Theorem \ref{thm:flatdeg}), where $SL_n/\Q{n}{m}$ is the space of partial flags 
$\{0\subsetneq V_1\subsetneq \dots \subsetneq V_n \subseteq \fk^m \mid \dim(V_i)=i\}$, 
which in turn implies that the affine variety $\mathcal{V}(\R{n}{m})$ is 
a toric degeneration of $SL_n/\Q{n}{m}$ (Corollary \ref{cor:toricdegeneration}). 
Both of these results are extensions of the aforementioned classical flat degeneration 
occurring when $m = n$. Knowing this, we use standard commutative algebra techniques 
to show that $\A{n}{m}$ is Cohen--Macaulay, and thus so is 
$\bigoplus_{\underline{a}}\Gamma(SL_m/\Q{n}{m},L^{\underline{a}})$. 
This flat degeneration can be used for the aforesaid enumeration 
of the minimal number of Grassmannian and incidence relations for the partial flag variety 
$SL_m/\Q{n}{m}$ (Theorem \ref{thm:enumeration} and Corollary \ref{cor:pluckerenumeration}). 
A particularly charming feature of our approach is that it both unifies 
and provides an explanation for the apparent discrepancy between 
the classical flat degeneration of the Pl{\"u}cker algebra and the flat degenerations 
of Gonciulea--Lakshmibai arising from the complete flag variety 
(Remark~\ref{rema:embedding-flag-variety}). 

Finally, as an application, we relate the multiplication of the tableaux algebra 
and the structure of crystal graphs for quantum group representations. 
Crystals were introduced independently by Kashiwara and Lusztig as 
a combinatorial skeleton for the irreducible highest weight modules of the quantum group 
of a complex Lie algebra $\mathfrak{g}$; see, 
\cite{kashiwara1990crystalizing, kashiwara1991crystal, Lusztig1}. 
For $\g = \mathfrak{sl}_n$, the basis of a crystal graph $\B(\lambda)$ is in bijection 
with the set of semistandard Young tableaux $\SSYT_n(\lambda)$. 
Since multiplication within $\A{n}{n}$ yields an action of $\SSYT_n^n$, 
for each $T \in \SSYT_n(\mu)$ we thus have an injective map of crystals 
$\Phi_T:\B(\lambda) \to \B(\lambda+\mu)$ sending a tableau $T' \in \SSYT_n(\lambda)$ 
to the tableau $T' \star T \in \SSYT_n(\lambda+\mu)$. 
In the context of string polytopes, see 
\cite{Stu95,conescrystalspatterns, BossingerFourier,BerensteinZelevinsky}, 
the analogous map was studied by Bossinger--Torres in \cite{BT25}. 
They showed that when $T$ corresponds to a weight zero vector in 
the adjoint representation, the map corresponding to $\Phi_T$ is 
a weight-preserving crystal embedding. We extend these results and show that 
when $T$ is a highest or lowest weight vector, the induced map $\Phi_T$ 
is also a crystal embedding (Theorem \ref{thm:crystal}). 

\subsection*{Outline} 

In Section \ref{sec:algebra}, we establish the preliminaries, 
as well as the main structural and enumerative results of the tableaux algebra 
and its defining ideal. In Section \ref{sec:ideals-spectra}, 
we determine the maximal spectrum of the tableaux algebra and describe its topology. 
In Section \ref{sec:rep}, we briefly study the representations of 
the tableaux algebra, including a classification of its irreducible modules, 
and providing a faithful semisimple module. In Section \ref{sec:toric}, 
we discuss some geometric connections between the tableaux algebra and 
partial flag varieties, including the Cohen--Macaulayness of the former and 
that the zero locus of the defining ideal of the former is a toric degeneration 
of the latter. In Section \ref{sec:crystals}, we conclude the paper 
by describing how the star operation on semistandard Young tableaux 
gives rise to various crystal embeddings. 

%%%%%%%%%%%%%%%%%%%%%%%%%%%%%%%%%%%%%%%%%%%%%%%%%%%%%%%%%%%%%%%%%%%%%%%%%%%%%%%%

\section*{Acknowledgments}

The authors thank the organizers of the 2025 Collaborative Workshop in 
Algebraic Combinatorics at the Institute for Advanced Study in Princeton, 
where this project began. 
We also thank Lara Bossinger, Hood Chatham, Bea de Laporte, 
Masato Kobayashi, Peter Littelmann, Peter McNamara, 
Kaveh Mousavand, Jos{\'e} Simental, Bernd Sturmfels, Frank Sottile, 
and Jerzy Weyman for helpful conversations. 
NG, BM, PSO, and JP were supported by the National Science Foundation 
under grant DMS-1929284, while 
at the Institute for Computational and Experimental Research 
in Mathematics in Providence, during the Categorification and Computation 
in Algebraic Combinatorics program. 
BM was supported by Narodowe Centrum Nauki, 
grants 2021/43/D/ST1/02290 and 2021/42/E/ST1/00162. 
PSO was supported by JSPS KAKENHI Grant Number JP25K17242. 
JT was supported by Narodowe Centrum Nauki, grant 2021/43/D/ST1/02290 and ID.UJ. 

%%%%%%%%%%%%%%%%%%%%%%%%%%%%%%%%%%%%%%%%%%%%%%%%%%%%%%%%%%%%%%%%%%%%%%%%%%%%%%%%

\section{The tableaux algebra}\label{sec:algebra}

\subsection{Preliminaries} 

Henceforth, fix $\fk$ to be an algebraically closed field of characteristic zero. 

A \newword{partition} of height $n$ is a tuple of nonnegative integers 
$\lambda = (\lambda_1,\dots,\lambda_n)$ such that $\lambda_i\geq \lambda_{i+1}$ for all 
$1 \leq i \leq n$, with $|\lambda| \coloneqq \sum_i \lambda_i$. 
As usual, we identify a partition with its Ferrers diagram 
(in English notation) of left justified boxes (or cells), whose $i^{th}$ row consists of 
$\lambda_i$ boxes. 
A \newword{semistandard Young tableaux of shape $\lambda$} is a filling of 
the diagram of $\lambda$ with positive integers that is strictly increasing 
down the columns, and weakly increasing right along the rows. 
We denote by $\SSYT_{m}^{n}$ the set of all semistandard Young tableaux of 
all possible partition shapes with at most $n$ rows and entries in $\{1,\ldots, m\}$, 
and by $\SSYT_{m}(\lambda)$ the set of all semistandard fillings of tableaux 
of fixed partition shape $\lambda$ with entries in $\{1,\ldots, m\}$. 
When the entries in a semistandard tableau of shape $\lambda$ are in bijection with $\{1,\dots,|\lambda|\}$ 
we call the tableau \newword{standard}. 
When no restrictions are imposed on the length of $\lambda$ or the maximal value 
of the fillings then we simply omit the corresponding label. 

\begin{definition}\label{staroperation}
    Given any semistandard Young tableaux $T$ of shape $\lambda$ 
    and $T'$ of shape $\mu$, define the \newword{star product} $T \star T'$ to be 
    the semistandard Young tableaux of shape $\lambda+\mu$ obtained by 
    horizontally concatenating the rows and then sorting the entries in each row 
    in weakly increasing order from left to right. 
\end{definition}

\begin{proposition}\label{prop:monoid}
    Let $m, n \in \mathbb{N}$ with $m \geq n$. 
    The triple $(\SSYT_m^n, \star, \emptyset)$ is a multigraded, 
    cancellative, commutative, reduced, torsion-free monoid, 
    generated by the set of all columns
    \[\G{n}{m} \coloneqq \{ T \in \SSYT_{m}^{n}(1^k), \; 1\leq k \leq n\}.\] 
\end{proposition}

\begin{proof}
    The multigrading is obvious. 
    Cancellative and commutative are immediate as 
    in~\cite[Lemma 3.2]{li:canonical}. 
    Since the star product weakly increases the number of columns, 
    $\emptyset$ is the only invertible element, so the monoid is reduced. 
    Torsion-free is easily checked by contrapositive, because two 
    distinct tableaux differ in at least one entry, hence their 
    successive star products will also differ in at least that same 
    entry. 
\end{proof}

\begin{definition}\label{def:rowweight}
    Given any $T \in \SSYT_m^n$, for each $1\leq i \leq n$ define the 
    \newword{row weights} $\wt^{(i)} (T)$ of $T$ to be the tuples 
    \[
        \wt^{(i)} (T)= \left(\wt_1^{(i)}(T), \dots,\wt_m^{(i)}(T)\right) \in \mathbb{Z}^m 
    \]
    where each $\wt_j^{(i)}(T)\in \mathbb{Z}_{\geq 0}$ counts 
    the number of $j$'s that appear in the $i^{th}$ row of $T$. 
\end{definition}

\begin{definition}\label{def:colreadingword}
    For $T \in \SSYT_n$ denote by $w_{col}(T)$ the
\newword{column reading word} of $T$, obtained by reading the entries of $T$ 
up the columns starting with the leftmost column and moving right. 
\end{definition}

\begin{example}
    Consider $T$ and $T'$ in $\SSYT_4^3$ below. Then, their star product $T \star T'$ is given by 
    \[T \star T'=
        \tableau{ 
            1 & 1 \\
            2 & 3 \\
            4 
        }
        \; \star \;
        \tableau{ 
            1 & 2 \\
            2 \\
            3 
        }
        \ \ = \ \  
        \tableau{ 
            1 & 1 & 1 & 2 \\
            2 & 2 & 3 \\
            3 & 4 
        }\;, 
    \]
    where $w_{col}(T)=42131$, $w_{col}(T')=3212$, $w_{col}(T \star T') = 321421312$, and 
    \begin{align*}
        \wt^{(1)}(T\star T') &= (3,1,0,0) = (2,0,0,0) + (1,1,0,0) 
        = \wt^{(1)}(T) + \wt^{(1)}( T')\\
        \wt^{(2)}(T\star T') &= (0,2,1,0) = (0,1,1,0) + (0,1,0,0) 
        = \wt^{(2)}(T) + \wt^{(2)}( T')\\
        \wt^{(3)}(T\star T') &= (0,0,1,1) = (0,0,0,1) + (0,0,1,0) 
        = \wt^{(3)}(T) + \wt^{(3)}( T').
    \end{align*}
\end{example}

\begin{remark}
    An equivalent definition of the star product $T\star T'$ is as 
    the unique semistandard Young tableaux 
    of shape $\lambda+\mu$ satisfying 
    $\wt^{(i)}(T\star T') = \wt^{(i)}(T)+\wt^{(i)}(T')$ for all $i$. 
\end{remark}

\begin{definition}\label{def:tableaux-monoid-algebra}
    Let $m, n \in \mathbb{N}$ with $m\geq n$. We define the 
    \newword{tableaux algebra} $\A{n}{m}$ as the monoid algebra of $\SSYT_m^n$. 
\end{definition}

Since $\SSYT_m^n$ is commutative, 
it is immediate that $\A{n}{m}$ is also commutative. 

\begin{remark}
    An equivalent definition for $\A{n}{m}$ is as the commutative, associative, unital $\fk$-algebra 
    generated by the set of tableaux $T \in \SSYT_m^n$ taken under formal sums, 
    where multiplication is given by $\star$ and the identity element is given by the empty tableau $\emptyset$. 
\end{remark}

The family of algebras $\A{n}{m}$ with $n$ and $m$ ranging over all admissible 
natural numbers admits filtrations corresponding to 
the inclusions $\A{n}{m} \hookrightarrow  \A{n+1}{m}$ 
and $\A{n}{m} \hookrightarrow \A{n}{m+1}$, 
and projections $\A{n}{m} \twoheadrightarrow \A{n-1}{m}$ 
and $\A{n}{m} \twoheadrightarrow \A{n}{m-1}$,
visualized below. 
\[
    \begin{tikzcd}
        \A{1}{1} \arrow[r,shift left,hookrightarrow]
        & 
        \A{1}{2} \arrow[r,shift left ,dotted,hookrightarrow] \arrow[l,shift left,twoheadrightarrow]
        \arrow[d,hookrightarrow,shift left] 
        & 
        \A{1}{n} \arrow[l,shift left, dotted,twoheadrightarrow]\arrow[r, shift left, dotted,hookrightarrow]  \arrow[d,hookrightarrow,shift left] 
        & 
        \A{1}{m} \arrow[l,shift left, dotted,twoheadrightarrow] \arrow[r,shift left,hookrightarrow]  \arrow[d,hookrightarrow,shift left]
        & 
        \A{1}{m+1} \arrow[l,shift left,twoheadrightarrow] \arrow[r,shift left,dotted,hookrightarrow]  \arrow[d,hookrightarrow,shift left]
        & \arrow[l,shift left,dotted,twoheadrightarrow] {} 
        \\
        {} 
        & \A{2}{2} \arrow[r,shift left ,dotted,hookrightarrow]\arrow[u,twoheadrightarrow,shift left]
        & \A{2}{n} \arrow[l,shift left,dotted,twoheadrightarrow]\arrow[r,shift left ,dotted,hookrightarrow] \arrow[u,twoheadrightarrow,shift left] \arrow[d,hookrightarrow,shift left,dotted] 
        & \A{2}{m} \arrow[l,shift left,dotted,twoheadrightarrow]\arrow[r,shift left,hookrightarrow] \arrow[u,twoheadrightarrow,shift left] \arrow[d,hookrightarrow,shift left,dotted] 
        & \A{2}{m+1} \arrow[l,shift left,twoheadrightarrow]\arrow[r,shift left ,dotted,hookrightarrow] \arrow[u,twoheadrightarrow,shift left] \arrow[d,hookrightarrow,shift left,dotted]
        & \arrow[l,shift left,dotted,twoheadrightarrow]{} 
        \\
        {}
        & {} 
        & \A{n}{n} \arrow[r,shift left ,dotted,hookrightarrow] \arrow[u,twoheadrightarrow,shift left,dotted] 
        & \A{n}{m} \arrow[l,shift left,dotted,twoheadrightarrow]\arrow[r,shift left ,hookrightarrow]  \arrow[u,twoheadrightarrow,shift left,dotted] \arrow[d,hookrightarrow,shift left] 
        & \A{n}{m+1} \arrow[l,shift left,twoheadrightarrow]\arrow[r,shift left ,dotted,hookrightarrow]  \arrow[u,twoheadrightarrow,shift left,dotted] \arrow[d,hookrightarrow,shift left]
        & \arrow[l,shift left,dotted,twoheadrightarrow]{} 
        \\
        {} 
        & {} 
        & {} 
        & \A{n+1}{m} \arrow[r,shift left,hookrightarrow] \arrow[u,twoheadrightarrow,shift left] \arrow[d,hookrightarrow,shift left,dotted]
        & \A{n+1}{m+1} \arrow[l,shift left,twoheadrightarrow]\arrow[r,shift left ,dotted,hookrightarrow]   \arrow[u,twoheadrightarrow,shift left] \arrow[d,hookrightarrow,shift left,dotted] 
        & \arrow[l,shift left,twoheadrightarrow,dotted]{} 
        \\
        {} 
        & {} 
        & {} 
        & \arrow[u,twoheadrightarrow,shift left,dotted] {} 
        & \arrow[u,twoheadrightarrow,shift left,dotted] {} 
        & {} 
    \end{tikzcd}
\]

The inclusions are straightforward. The projections are given by linearly extending the following assignment 
on tableaux. 
\begin{equation*}
    \begin{tikzcd}[row sep=0]
        \A{n}{m} \ar[r,two heads] 
        & \A{n-1}{m}\\
        T \ar[r, mapsto] 
        & \begin{cases} 
            T \text{ if } \wt_j^{(n)}(T) = 0 \ \forall j,\\
            0 \text{ otherwise}.
        \end{cases} 
    \end{tikzcd}
    \quad
    \begin{tikzcd}[row sep=0]
        \A{n}{m} \ar[r,two heads] 
        & \A{n}{m-1}\\
        T \ar[r, mapsto] 
        & \begin{cases} 
            T \text{ if } \wt_m^{(i)}(T) = 0 \ \forall i,\\
            0 \text{ otherwise}.
        \end{cases} 
    \end{tikzcd}
\end{equation*}
Although it will not play a role in this paper, it is worth mentioning that 
the inverse and direct limits of the above diagrams 
of algebra morphisms coincide. 
The resulting algebra $\mathcal{T} \coloneqq \varprojlim{\A{n}{m}} \cong 
\varinjlim{\A{n}{m}}$ is of independent interest. 

\begin{theorem}\label{thm:finite generation}
    The algebra $\A{n}{m}$ is finitely generated over $\fk$ with minimal 
    generating set given by $\G{n}{m}$, where 
    $|\G{n}{m}| = \sum_{k=1}^n  {\binom{m}{k}}$. 
\end{theorem}

\begin{proof}
    The monoid $\SSYT_m^n$ is generated by columns which form a minimal 
    generating set because the product of two columns has strictly more 
    than one column. 
    Since $\A{n}{m}$ 
    is the monoid algebra of $\SSYT_m^n$, so that every element in 
    $\A{n}{m}$ is a linear combination 
    in elements 
    of $\SSYT_m^n$, then the columns $\G{n}{m}$ form a minimal generating 
    set for $\A{n}{m}$. Lastly, semistandard Young tableaux are strictly 
    increasing on columns from which the enumeration follows. 
\end{proof}

\subsection{Algebraic structure}

The tableaux algebra inherits many desirable properties 
by virtue of Theorem~\ref{thm:finite generation}. 

\begin{proposition}\label{prop:IntegralDomain-Noetherian-Reduced-Jacobson}
    For any positive integers $m\geq n$, the tableaux algebra $\A{n}{m}$ is 
    an integral domain. Moreover, as a ring $\A{n}{m}$ is also 
    Noetherian, reduced, and Jacobson. 
\end{proposition}

\begin{proof}
    Recall that $\SSYT_m^n$ is a cancellative, commutative, and torsion-free monoid 
    by Proposition~\ref{prop:monoid}. Thus, its monoid algebra $\A{n}{m}$ is an 
    integral domain by~\cite[Theorem 8.1]{gilmer:semigroup}. 
    As an immediate consequence, $\A{n}{m}$ is indeed reduced. 
    Since $\A{n}{m}$ is finitely generated by Theorem~\ref{thm:finite generation}, 
    it is a quotient of a polynomial ring, which is Noetherian. 
    As a quotient of a Noetherian ring, 
    $\A{n}{m}$ is also Noetherian. 
    Lastly, since $\A{n}{m}$ is finitely generated over a field 
    by Theorem~\ref{thm:finite generation}, and any field is a Jacobson ring, 
    recalling that any finitely generated ring over a Jacobson ring is itself 
    Jacobson yields that $\A{n}{m}$ is Jacobson. 
\end{proof}

In particular, the fact that $\A{n}{m}$ is Jacobson implies that all prime ideals arise 
as intersections of maximal ideals. More precisely, a prime ideal coincides with the 
intersection of the maximal ideals containing it. 
We will exploit this in Section~\ref{sec:prime-ideals-tableau-algebra}. 

\begin{corollary}\label{cor:jacobson}
    The Jacobson radical of $\A{n}{m}$ is zero.
\end{corollary}

\begin{proof}
    Recall that any commutative and finitely generated algebra over a field 
    has Jacobson radical equal to its nilradical. 
    Since $\A{n}{m}$ is finitely generated over a field by 
    Theorem~\ref{thm:finite generation}, and it has nilradical equal to zero 
    because it is reduced by 
    Proposition~\ref{prop:IntegralDomain-Noetherian-Reduced-Jacobson}, 
    the result follows. 
\end{proof}

The tableaux algebra has a curious behavior. 
Having zero Jacobson radical, $\A{n}{m}$ is semiprimitive. 
Being commutative and an integral domain but not a field, 
$\A{n}{m}$ is not primitive and not Artinian, respectively. 
Thus, as a non-Artinian semiprimitive ring, $\A{n}{m}$ is not semisimple. 
However, being semiprimitive, we should still be able to understand it in terms 
of its irreducible modules. We will discuss this in detail in Section \ref{sec:rep}.  

\subsection*{Total ordering on tableaux} 

We now define a total order on the set $\SSYT_m^n$ by first comparing 
the shapes of the partitions and then comparing their entries. 
Given any column tableau $C$ of shape $(1^k)$, let $\height(C)=k$ be 
the height of $C$. For any two column tableaux $C$ and $C'$, we say that 
$C$ is less than $C'$ if either $\height(C)<\height(C')$, 
or $\height(C)=\height(C')$ and $w_{col}(C)<w_{col}(C')$ in lexicographic order on words. 
Two columns are equal if and only if they have the same height and column reading word. 
So then, given any two tableaux $T$ of shape $\lambda=(\lambda_1, \dots, \lambda_r)$ 
with columns $C_i$ and $T'$ of shape $\mu=(\mu_1,\dots,\mu_s)$ with columns $C_i'$, 
we say that $T$ is less than $T'$ whenever 
\begin{enumerate}
    \item $\lambda_1<\mu_1$, or 
    \item $\lambda_1=\mu_1$ and there exist a $j$ such that 
        $C_i=C_i'$ for all $i < j$ and $C_j$ is less than $C_j'$. 
\end{enumerate}

We denote that $T$ is less than or equal to $T'$ by $T \preceq T'$. 
It is straightforward to check that this is indeed a total order. 
The lexicographic order induced by $\preceq$ 
on tuples of elements of $\SSYT_m^n$ will also be denoted by $\preceq$, 
as per the usual abuse of notation. 

\begin{remark}\label{rem:roworder}
    We observe that an analogous total order on tableaux can be defined by 
    comparing the rows instead of the columns of the tableaux. 
    In particular, as described in Remark~\ref{rema:compatible-total-orders}, 
    this row order is compatible with the 
    upcoming interpretation of the tableaux algebra as a subalgebra 
    of a polynomial ring. Translating the column order from tableaux 
    to polynomials is more involved, but as it is more natural 
    given our presentation of $\A{n}{m}$, we choose to use it instead. 
\end{remark}

Recall that an algebra $A$ is \newword{quadratic} when it 
is finitely generated and all its relations are homogeneous of degree two. 
That is, $A$ is of the form $T(V)/\langle R \rangle$ where 
$V$ is a vector space and $R  \subseteq V \otimes V$. 

With this in mind, we set 
$\tensor[_{n\hspace{-1mm}}]{V}{_m} \coloneqq \Span_{\fk}\{ T \in \G{n}{m}\}$, 
with $\G{n}{m}$ as in Theorem \ref{thm:finite generation}. 

\begin{theorem}\label{thm:quadratic}
    The algebra $\A{n}{m}$ is quadratic. 
\end{theorem}

\begin{proof}
    Clearly $\A{n}{m} = T(\ \tensor[_{n\hspace{-1mm}}]{V}{_m})/\langle R \rangle$, where 
    $T(\ \tensor[_{n\hspace{-1mm}}]{V}{_m}) = \bigoplus_{k\geq 0} (\ \tensor[_{n\hspace{-1mm}}]{V}{_m})^{\otimes k}$ 
    is the tensor algebra and $\langle R \rangle$ is a set of relations. 
    The relations in $R$ are of the form $T_1 \otimes T_2 - T_2 \otimes T_1$ and 
    $T_1 \otimes T_2 - T_3 \otimes T_4$ for some distinct 
    $T_i \in \tensor[_{n\hspace{-1mm}}]{V}{_m}$ 
    by definition. In particular, 
    all the relations in $R$ are homogeneous of degree two, 
    so $\A{n}{m}$ is quadratic. 
\end{proof}

Although all Koszul algebras are quadratic, the converse is not true. 
However, our quadratic tableaux algebra is indeed Koszul. 

\begin{corollary}\label{cor:koszul}
    The algebra $\A{n}{m}$ is Koszul. 
\end{corollary}

\begin{proof}
    Within $\A{n}{m}$ the relations become the following. 
    \begin{align} 
        T_1 \star T_2 - T_2 \star T_1 
        &= 0 
        \qquad \label{eq:comm-relation}\text{(commutation relation)}\\ 
        T_1 \star T_2 - T_3 \star T_4 
        &= 0 
        \qquad \label{eq:prod-relation}\text{(product relation)} 
    \end{align}
    Since elements in $\G{n}{m}$ admit a total order, the algebra admits a 
    presentation of the relations above where $(T_1,T_2) \succ (T_2,T_1)$, 
    $(T_3,T_4) \succ (T_4,T_3)$, and $(T_1,T_2) \succ (T_3,T_4)$ for any 
    $T_i \in \G{n}{m}$ that satisfy them. 
    Since the generators $\G{n}{m}$ of $\A{n}{m}$ are linearly independent, 
    this means $\A{n}{m}$ is a PBW algebra, and thus Koszul by 
    \cite[Theorem 4.3.1]{polishchuk2005}. 
\end{proof}

\begin{remark}\label{rema:koszul-combinatorial}
    This is far from the only available proof of the Koszulity of $\A{n}{m}$, 
    as once we establish Theorem~\ref{thm:Polynomial-Iso} we could appeal 
    to \cite[Corollary 4.1]{hrw1998koszul}, for example. 
    Nonetheless, a significant advantage of our proof 
    is its combinatorial and elementary nature. 
\end{remark}

Every element in $T \in \G{n}{m}$ satisfies equation~\eqref{eq:comm-relation} 
for any choice of column $T'$ because $\A{n}{m}$ is commutative. 
However, as we will see, not every  $T \in \G{n}{m}$ satisfies a relation 
of the form~\eqref{eq:prod-relation}. 
In fact, it will be very useful to give a precise description of which elements 
of $\G{n}{m}$ satisfy~\eqref{eq:comm-relation} but not~\eqref{eq:prod-relation}. 
The following definition has that in mind. 

\begin{definition}\label{def:vector-space-quadratic}
    Let $m, n \in \mathbb{N}$ with $m\geq n$. Set 
    \[\E{n}{m} \coloneqq \G{n}{m} \setminus \F{n}{m},\]
    where $\F{n}{m}$ is one of the following subsets of $\G{n}{m}$: 
    \begin{equation*} 
        \F{n}{n} \coloneqq 
        \left\lbrace \;
        \hackcenter{
            \tableau{n} \; , 
            \tableau{
                1\\ 
                2\\ 
                \vdots\vspace{2mm}\\
                {\text{\small{$n$--1}}}
            } \; , 
            \tableau{
                1\\ 
                2\\
                \vdots\vspace{2mm}\\
                {\text{\small{$n$--1}}}\\ 
                n
            }
        } \;
        \right\rbrace \; (m = n), 
        \qquad 
        \F{n}{m} \coloneqq 
        \left\lbrace \; 
        \hackcenter{ 
            \tableau{m} \; , 
            \tableau{
                1\\ 
                2\\ 
                \vdots\vspace{2mm}\\ 
                {\text{\small{$n$--1}}}\\ 
                n
            }
        } \;
        \right\rbrace \;(m \neq n). 
    \end{equation*}
\end{definition}

\begin{lemma}\label{lem:prod-ideal}
    A column $T \in \A{n}{m}$ satisfies a product relation 
    of the form \eqref{eq:prod-relation} with some distinct 
    $T',T'',T''' \in \G{n}{m}$ if and only if $T \in \E{n}{m}$. 
\end{lemma}

\begin{proof}
    We prove the following equvalent statement; a column 
    $T \in \A{n}{m}$ does not satisfy 
    any product relation if and only if $T \in \F{n}{m}$. 
    Suppose first that $T$ is a column of height $1$ with reading word $w_{col}(T)=a$ 
    for some $1\leq a \leq m$. 
    Since for any $a<m$ there exists $b>a$ satisfying 
    \[
       \tableau{a} \;\star\; \tableau{1\\b} = \tableau{1} \;\star \;\tableau{a\\b} 
    \]
    then $T$
    satisfies a product relation 
    if and only if $a\neq m$. 

    Suppose now that $T$ is a column of height $k$ with 
    $w_{col}(T)=a_k\dots a_2a_1$ for some $1\leq k \leq n$. 
    If $a_1 \neq 1$ then 
    \[  \tableau{a_1\\a_2\\\vdots\vspace{2mm}\\a_k} \;\star\; \tableau{1} 
       = \tableau{1\\a_2\\\vdots\vspace{2mm}\\a_k} \;\star\; \tableau{a_1} 
    \]
    and $T$ satisfies a product relation. If $a_1=1$ 
    and $a_2\neq 2$ we can repeat 
    the argument to find a column $T'\prec T$ satisfying 
    \[T \star \hackcenter{\tableau{1\\2}} = T' \star \hackcenter{\tableau{1\\a_2}}.\]
    Continuing in this manner we see that any nonstandard filling $T$ 
    will satisfy a product relation.
    Thus, let $T$    
    be a standard column tableau
    of height $1\leq k\leq n$ 
    and consider the relation 
    \[
       \tableau{
           1\\
           2\\
           \vdots\vspace{2mm}\\
           {\text{\small{$k$--1}}}
           \\k
       } \;
       \star\; 
       \tableau{ 
           1\\
           2\\
           \vdots\vspace{2mm}\\
           {\text{\small{$k$--1}}}\\
           {\text{\small{$k$+1}}}\\
           {\text{\small{$k$+2}}} 
       } 
       =
       \tableau{ 
           1\\
           2\\
           \vdots\vspace{2mm}\\
           {\text{\small{$k$--1}}}\\
           {\text{\small{$k$+1}}} 
       } 
      \; \star \;
       \tableau{ 
           1\\
           2\\
           \vdots\vspace{2mm}\\
           {\text{\small{$k$--1}}}\\
           k\\
           {\text{\small{$k$+2}}} 
       }.
    \]
    For $m>n$, this implies $T$ satisfies a product relation for all $1<k<n$. 
    If instead $m=n$, then the relation above can only hold whenever $k<n-1$. 
    In fact, when $k=n-1$ the consecutive entries of $T$ prevent any such 
    product relations from occurring 
    since the only entry that could potentially be swapped 
    is the lowest one. However, the only way of exchanging $n-1$ with $n$ 
    in this case yields 
    a commutation relation, not a product relation. 
    Similarly, for any $m\geq n$, 
    if $k=n$ then once again the consecutive entries of $T$ prevent any swaps 
    from being possible. 
    Thus, the only columns in $\A{n}{m}$ not subject to any 
    product relations are precisely those in $\F{n}{m}$. 
\end{proof}

To simplify notation, we will employ the following convention. 

\begin{definition}
    For a finite set $Y = \{y_1,\dots,y_k\}$ denote by 
    $\fk[Y]$ the polynomial ring $\fk[y_1,\dots,y_k]$. 
    Let $f, g \in \fk[Y]$, we denote their product by the 
    concatenation $fg \in \fk[Y]$. 
\end{definition}

We will now consider the ideal of relations of the tableaux algebra. 

\begin{definition}
    Let $\R{n}{m}$ be the ideal in $\fk[\G{n}{m}]$ generated by 
    all product relations of the form $T_1\star T_2 - T_3 \star T_4$ for all distinct $T_i \in \G{n}{m}$. 
\end{definition}

In particular, $\R{n}{m}$ is a finitely generated ideal. 

\begin{corollary}\label{cor:tensorprod}
    The ideal $\R{n}{m}$ is contained in $\fk[\E{n}{m}]$, hence 
    there is an algebra isomorphism 
    \[
        \A{n}{m} \cong \fk[\E{n}{m}]/\R{n}{m} \otimes \fk[\F{n}{m}]. 
    \]
    Moreover, $\R{n}{m}$ is prime in both $\fk[\E{n}{m}]$ and $\A{n}{m}$.
\end{corollary}

\begin{proof} 
    Since $\A{n}{m}$ is generated by $\G{n}{m}$ 
    we know from the proof Theorem \ref{thm:quadratic} 
    that $\A{n}{m} = T(\ \tensor[_{n\hspace{-1mm}}]{V}{_m}) / \langle R \rangle
    \cong \fk[\G{n}{m}] / \mathcal{P}$ for some ideal $\mathcal{P}$ 
    consisting of all product relations. 
    We know that the only elements in 
    $\G{n}{m}$ subject to such relations are those in $\E{n}{m}$ 
    by Lemma \ref{lem:prod-ideal}, 
    so $\R{n}{m}$ is an ideal in $\E{n}{m}$. 
    Since $\F{n}{m} = \G{n}{m} \setminus \E{n}{m}$, the isomorphism follows. 
    Lastly, $\fk[\E{n}{m}]/\R{n}{m}$ is an integral domain because it is 
    a unital subring of the integral domain $\A{n}{m}$, 
    so $\R{n}{m}$ is prime. 
\end{proof}

Setting $\fE{n}{m} \coloneqq \fk[\E{n}{m}]/\R{n}{m}$ and $\fF{n}{m} \coloneqq \fk[\F{n}{m}]$ 
then, by Corollary~\ref{cor:tensorprod}, we obtain
\begin{equation} \label{eq:TensorDecompose}
\A{n}{m} \cong \fE{n}{m} \otimes \fF{n}{m}.    
\end{equation} 
We will use this decomposition further when we study its spectra in 
Section~\ref{sec:ideals-spectra}. 

\begin{example}
    Consider $\A{3}{3}$, the generators in $\G{3}{3}$ are in the following order. 
    \[
        \hackcenter{\tableau{1}}\;\prec\; 
       \hackcenter{ \tableau{2}}\;\prec\; 
        \hackcenter{\tableau{3}}\;\prec\; 
        \hackcenter{\tableau{1\\2}}\;\prec\; 
       \hackcenter{ \tableau{1\\3}}\;\prec\; 
        \hackcenter{\tableau{2\\3}}\;\prec\; 
        \hackcenter{\tableau{1\\2\\3}} 
    \]
    It is easy to see that in $T(\ \tensor[_{3\hspace{-1mm}}]{V}{_3})$ the only 
    product relation among the generators is 
    \[
        \tableau{2\\3}\;\otimes \;\tableau{1} 
        - \tableau{1\\3}\;\otimes \;\tableau{2} \; 
        \in R 
    \]
    or equivalently that 
    \begin{equation}\label{eq:first-product-relation}
        \tableau{2\\3} \;\star \;\tableau{1} = \tableau{1\\3} \;\star\; \tableau{2}
    \end{equation}
    in $\A{3}{3}$. 
    Indeed, the only generators not present in this relation are exactly those 
    in 
    \[\F{3}{3} = \left\lbrace \;\hackcenter{\tableau{3}\;, \tableau{1\\2}\;, \tableau{1\\2\\3}} \;
    \right\rbrace.\]
    Now consider instead $\A{2}{3}$. Although the only product relation 
    is again \eqref{eq:first-product-relation}, $\G{2}{3}$ does not contain 
    the last column of $\G{3}{3}$, and thus the generators 
    not present in \eqref{eq:first-product-relation} are those in the set 
    \[\F{2}{3} = \left\lbrace \;\hackcenter{\tableau{3}\;, \tableau{1\\2}} \; \right\rbrace.\] 
\end{example}

\begin{corollary}
    For any positive integer $n$ we have $\R{n-1}{n}=\R{n}{n}$. 
\end{corollary}

\begin{proof}
    Since $\G{n}{n}\setminus \G{n-1}{n} = \F{n}{n}\setminus \F{n-1}{n}$, 
    both containing only the standard tableau of height $n$, 
    the result follows from Corollary~\ref{cor:tensorprod}. 
\end{proof}

\subsection{Enumerating relations}\label{subsec:enumeration}

We now give an upper bound 
on the codimension of the variety $\mathcal{V}(\R{m}{n})$ by enumerating 
the size of the minimal generating set for the ideal $\R{n}{m}$. 
We will further explore its geometric meaning 
in Section~\ref{sec:toric}. 

\begin{definition}
    Let $\varsigma(\R{n}{m})$ denote 
    the cardinality of a minimal generating set for the ideal $\R{n}{m}$. 
\end{definition}

To count $\varsigma(\R{n}{m})$ consider the following set up. 
For $S \in \SSYT_m^n$ any two column tableau, let
\[
    \COL(S)\coloneqq 
    \{ (T,T') \in \G{n}{m} \times \G{n}{m} \mid T \star T' = S \} 
\]
be the set of all pairs $T,T' \in \G{n}{m}$ such that $S=T \star T'$. 
For any pair $(T,T') \in \COL(S)$ set
\[
L_{T,T'} \coloneqq \text{ left column of }S \qquad \text{ and } \qquad R_{T,T'} \coloneqq \text{ right column of }S.
\]
We endow $\COL(S)$ with the following total order. 
Given $(T,T'), (T'', T''') \in \COL(S)$, we say 
$(T,T') \lhd (T'', T''')$ when either: 
\begin{enumerate}
    \item $\height(T)>\height(T'')$, or 
    \item $\height(T)=\height(T'')$ and $w_{col}(T)<w_{col}(T'')$. 
\end{enumerate}
In particular, this order is \emph{not} equivalent to the lexicographic order 
on tuples induced by $\preceq$ on tableaux. 
It is easy to see that under this order the minimal element of $\COL(S)$ 
is precisely the tuple $(L_{T,T'},R_{T,T'})$ where $(T,T')$ is any pair in $\COL(S)$. 

Observe that we can obtain the minimal number of product relations 
$T\star T' = S = T''\star T'''$ yielding $S$ 
from the number of relations involving the minimal element of $\COL(S)$. 
More precisely, the minimal number of product relations yielding $S$ coincides 
with the number of tuples  $\Big((L_{T,T'},R_{T,T'}),(T,T')\Big)$ 
for which  
$(L_{T,T'},R_{T,T'}) \lhd (T,T')$ and 
$(L_{T,T'},R_{T,T'}) \neq  (T,T')$.
Since the minimal element of $\COL(S)$ is unique, 
this is equivalent to counting the number of pairs $(T,T')$ 
where either: 
\begin{enumerate}
    \item $\height(T)>\height(T')$,
    $L_{T,T'} \neq T$, and $R_{T,T'} \neq T'$, or 
    \item $\height(T)=\height(T')$
    $L_{T,T'} \neq T$, and 
    $w_{col}(T)<w_{col}(T')$.
\end{enumerate}
Thus, $\varsigma(\R{n}{m})$ is equal to the number of pairs $(T,T')$ 
where either $T$ is longer than $T'$, or $T$ and $T'$ have the same height 
and $w_{col}(T)<w_{col}(T')$, minus the minimal element of $\COL(S)$ as 
$S$ ranges over all two column tableaux in
$\SSYT_m^n$. 

Given $T,T' \in \G{n}{m}$,  denote by $T\cap T'$ the set of repeated entries in $T$ and $T'$. 
To count $\varsigma(\R{n}{m})$ we will use the 
following auxiliary sets. 
For $i, j, k \in \mathbb{Z}_{\geq0}$ and $K \subseteq \{1, \dots, m\}$ with $|K| = k$, define: 
\begin{align*}
\ZZ{n}{m}(i,j,k)& \coloneqq 
\begin{cases}
\{(T,T')\in \G{n}{m} \times \G{n}{m}\mid \height(T)=\height(T')=i, \text{ and } w_{col}(T)<w_{col}(T')\}&; j=0
\\
\{ (T,T') \in \G{n}{m} \times \G{n}{m}\mid \height(T)=i+j,\, \height(T')=i,\text{ and } |T\cap T'|=k\} &; j>0,
\end{cases}
\\
\W{n}{m}(i,j,K)& \coloneqq \{ (T,T') \in \ZZ{n}{m}(i,j,k) \mid (L_{T,T'},R_{T,T'}) = (T,T'), \text{ and } T\cap T'=K\},
\\
\U{n}{m}(i,j,K)& \coloneqq \{ (T,T') \in \W{n}{m}(i-k,j,\emptyset) \mid T\cap K=T'\cap K=\emptyset\}.
\end{align*}

\begin{lemma}\label{lemm:total-number-columns}
    For any positive integers $m \geq n\geq i+j$, the cardinality of $\ZZ{n}{m}(i,j,k)$ is 
    \[
        \frac{1}{2} \binom{m}{k} \binom{m-k}{2i-2k} \binom{2i-2k}{i-k} 
        \text{ if } j=0 
        \quad \text{and} \quad 
        \binom{m}{k} \binom{m-k}{2i+j-2k} \binom{2i+j-2k}{i-k} 
        \text{ if } j>0.
    \]
\end{lemma}

\begin{proof}\label{lemm:proportion-number-columns}
    There are $\binom{m}{k}$ ways of choosing the $k$ repeated entries. 
    When $j >0$ there are $\binom{m-k}{2i+j-2k}$ ways of choosing 
    the remaining entries, and  $\binom{2i+j-2k}{i-k}$ ways of choosing 
    which entries go in the shorter column of height $i$. 
    When $j=0$ there are $\binom{m-k}{2i-2k}$ ways of choosing the entries 
    that only appear once, and $\frac{1}{2}\binom{2i-2k}{i-k}$ ways 
    of choosing $i-k$ entries for the first column such that 
    it is lexicographically smaller than the second column. 
\end{proof}

Given any possible set $K$ of cardinality $k$, we now count the elements in $\ZZ{n}{m}(i,j,k)$ 
satisfying certain ordering condition that are not 
the minimal element. 
To do this, we count the proportion of minimal elements and subtract it 
from $1$, yielding the proportion we desire. 

\begin{lemma}
    For $m \geq n$, the proportion of pairs 
    $(T,T') \in \ZZ{n}{m}(i,j,k)$ satisfying the condition that  
    $(T'',T''') \in \COL(T \star T')$ 
    implies $(T,T') \lhd (T'',T''')$, is 
    \[
        \frac{i-k-1}{i-k+1} 
        \text{ if } j=0 
        \quad \text{and} \quad 
        \frac{i-k}{i+j-k+1} 
        \text{ if } j\geq 1 
    \]
    of the cardinality of $\ZZ{n}{m}(i,j,k)$. 
\end{lemma}

\begin{proof}
    Let $K \subseteq \{1, \dots, m\}$ be the set of $k$ repeated entries, so 
    $|K| = k$. 
    Our task at hand is to count $\W{n}{m}(i,j,K)$, for all $K$. 
    There is a bijection from $\W{n}{m}(i,j,K)$ to $\U{n}{m}(i,j,K)$ 
    given by removing all of the entries 
    in $K$ from a pair of columns in $\W{n}{m}(i,j,K)$, 
    so $|\W{n}{m}(i,j,K)| = |\U{n}{m}(i,j,K)|$, and we now 
    count the latter. 

    We have $2i+j-2k$ entries, 
    where the $i+j-k$ entries in the first column have to be strictly 
    increasing in each row when compared to the $i-k$ entries in the second 
    column. 
    This corresponds to the number of words of length $2i+j-2k$ 
    having exactly $i+j-k$ ones and exactly $i-k$ twos such that 
    when read from left to right we always read at least as many ones as twos. 
    This is a case of the classical 
    weak ballot counting problem~\cite{Mac1909}. 
    When $j \geq 1$ it is counted by $\frac{(i+j-k)-(i-k)+1}{(i+j-k)+1}$ 
    of the total number of all binary words with $2i+j-2k$ entries. 
    Subtracting 
    it from $1$ establishes the second claim. When $j = 0$ it is counted by 
    $\frac{1}{i-k+1}$ of the total number of all binary words 
    with $2i-2k$ entries. 
    However, only half of all binary words with $2i-2k$ entries
    correspond to elements in $\U{n}{m}(i,0,K)$, 
    namely the ones 
    corresponding to the first column being less than the second in 
    lexicographic order. Thus, the proportion of columns we are interested 
    in is double of the aforementioned, namely $\frac{2}{i-k+1}$ 
    and subtracting 
    it from $1$ establishes the first claim. 
\end{proof}

\begin{theorem}\label{thm:enumeration}
    For $m \geq n$ the following equality holds, 
    \begin{equation*}
        \varsigma(\R{n}{m}) = \sum_{i=1}^{ n } 
        \sum_{\substack{k= \max \\ \{ 0,  2i-m \}}}^{i-1} 
        \sum_{j=0}^{\substack{\min\{n-i, \\ m-2i+k\}}} 
        \binom{m}{k}\left(1-\frac{1}{2}\delta_{0,j} \right) 
        \frac{i-k-\delta_{0,j}}{i+j-k+1}\binom{m-k}{2i+j-2k}\binom{2i+j-2k}{i-k}. 
    \end{equation*}
\end{theorem}

\begin{proof}
    For a fixed triple $i$, $j$, $k$ we have computed the total number 
    of pairs of columns in Lemma~\ref{lemm:total-number-columns}, and 
    the proportion of which we are interested in 
    Lemma~\ref{lemm:proportion-number-columns}. 
    It only remains to sum over all possible triples.
    Since $i$ ranges from $1$ to $n$, 
    for any fixed $i$, the value of $k$ can vary from the maximum of $0$ and $2i-m$ to 
    $i-1$. Furthermore, for any fixed $i$ and $k$, the value of $j$ ranges from 
    $0$ to the minimum of $n-i$ and $m-2i+k$. Thus, we obtain 
    \begin{align*}
        \varsigma(\R{n}{m}) &= \sum_{i=1}^{n}  \sum_{k= \max \{ 0, 2i-m \}}^{i-1} 
        \Bigg[ \left( \frac{1}{2} \right) \frac{i-k-1}{i-k+1} 
        \binom{m}{k} \binom{m-k}{2i-2k} \binom{2i-2k}{i-k} \\ 
        &+ \sum_{j=1}^{\min\{n-i,m-2i+k\}} \frac{i-k}{i+j-k+1} \binom{m}{k} 
        \binom{m-k}{2i+j-2k}\binom{2i+j-2k}{i-k} \Bigg]. 
    \end{align*}
    This expression can be simplified further 
    using the Kronecker delta $\delta_{i,j}$. 
\end{proof}

\begin{corollary}\label{cor:enumeration}
    For any positive integers $m\geq n$ we have 
    \begin{align*}
        \varsigma(\R{n}{m}) = 
        \sum_{i=1}^{ n }\binom{m}{i} \sum_{k = 1}^{\min\{i,m-i\}} 
        \binom{i}{k} 
        \Bigg[ - 
        \frac{1}{2} \binom{m-i}{k} + \sum_{j=k}^{\min\{n+k-i,m-i\}} \frac{k}{j+1} 
        \binom{m-i}{j} 
        \Bigg]. 
    \end{align*}
\end{corollary}

\begin{proof}
    The claim follows from repeated applications of the relation 
    $\binom{n}{k}\binom{k}{j} = \binom{n}{j}\binom{n-j}{k-j}$. 
\end{proof}

In particular, we have the following consequence of the above reasoning. 

\begin{corollary}
    Let $\BB{n}{m}$ be the subset of $\R{n}{m}$ consisting of the elements 
    \[
        T T' - L_{T,T'} R_{T,T'} 
    \]
    where either 
    \begin{enumerate}
        \item $\height(T)>\height(T')$
        , $L_{T,T'} \neq T$, and $R_{T,T'} \neq T'$, or 
        \item $\height(T)=\height(T')$, $L_{T,T'} \neq T$, and 
        $w_{col}(T)<w_{col}(T')$.
    \end{enumerate}
    Then $\BB{n}{m}$ is a minimal generating set for the ideal $\R{n}{m}$. 
\end{corollary}

%%%%%%%%%%%%%%%%%%%%%%%%%%%%%%%%%%%%%%%%%%%%%%%%%%%%%%%%%%%%%%%%%%%%%%%%%%%%%%%%

\section{Ideals and maximal spectrum}\label{sec:ideals-spectra}

\subsection{Maximal ideals} 

In this section we give a description of the maximal ideals of $\A{n}{m}$. 

To avoid confusion, given a finitely generated ideal $I$ of a ring $R$ 
with generators $r_1,\dots, r_k \in R$, we will write 
$I = \langle r_1,\dots,r_k \rangle_R$ to emphasize that the ideal is 
within $R$. We denote by $\mathcal{V}(I)$ the zero locus of $I$. 
In particular, $\mathcal{V}(\ \R{n}{m})$ is an affine algebraic variety 
by Corollary~\ref{cor:tensorprod}. Recall from \eqref{eq:TensorDecompose} that
$\A{n}{m} \cong \fE{n}{m} \otimes \fF{n}{m}$ with 
$\fE{n}{m} \coloneqq \fk[\E{n}{m}]/\R{n}{m}$ and $\fF{n}{m} \coloneqq \fk[\F{n}{m}]$. 

\begin{lemma}\label{lem:maximal}
    Suppose $\E{n}{m} = \{T_1,\dots,T_s\}$. 
    The maximal ideals of $\fE{n}{m}$ are of the form 
    \[
        {M}_{\underline{t}} 
        = \langle T_1-t_1,\dots,T_s-t_s\rangle_{\fE{n}{m}} 
    \]
    where 
    $\underline{t}=(t_1,\dots,t_s) \in \mathcal{V}(\ \R{n}{m}) 
    \subseteq \fk^{|\E{n}{m}|}$. 
\end{lemma}

\begin{proof}
    By the correspondence theorem, maximal ideals $M$ in $\fE{n}{m}$ are in bijection with 
    maximal ideals $\widetilde{M}$ in $\fk[\E{n}{m}]$ containing $\R{n}{m}$. 
    Since $\fk[\E{n}{m}] = \fk[T_1,\dots,T_s]$ is a polynomial ring, 
    its maximal ideals are of the form 
    \[
        \widetilde{M}_{\underline{t}} 
        = \langle T_1-t_1,\dots, T_s-t_s \rangle_{\fk[\E{n}{m}]} 
    \]
    where $\underline{t} \in \fk^{|\E{n}{m}|}$. 
    Since $\R{n}{m} \subseteq \widetilde{M}_{\underline{t}}$ if and only if 
    $\underline{t} \in \mathcal{V}(\ \R{n}{m})$, 
    then
    \[
        {M}_{\underline{t}} 
        = \widetilde{M}_{\underline{t}} / \R{n}{m} 
        = \langle T_1-t_1,\dots, T_s-t_s \rangle_{\fE{n}{m}}. 
        \eqno\qed 
    \]
    \hideqed
\end{proof}

In particular, Lemma \ref{lem:maximal} implies that maximal ideals in 
$\fE{n}{m}$ are in bijection with points $\underline{t} \in \fk^{|\E{n}{m}|}$ 
that satisfy the relations imposed by $\R{n}{m}$. 

\begin{theorem}\label{thm:Maxideals}
    Suppose $\E{n}{m}=\{T_1,\dots,T_s\}$ and $\F{n}{m}=\{T'_1,\dots,T'_r\}$
    with $r \in \{2,3\}$.
    The maximal ideals of $\A{n}{m}$ are of the form 
    \begin{equation}\label{eq:maxideal}
        M_{\underline{t}} 
        \coloneqq \langle T_1-t_1,\dots, T_s-t_s , T'_1-t'_1,\dots, T'_r-t'_r \rangle_{\A{n}{m}} 
    \end{equation}  
    where 
    $\underline{t} = (t_1, \dots, t_s, t'_1, \dots, t'_r) 
    \in \mathcal{V}(\R{n}{m}) \times \fk^{|\F{n}{m}|}$. 
\end{theorem}

\begin{proof}
    The tensor product of two algebras 
    $A_1 \otimes A_2$ has all maximal ideals 
    of the form $M_1 \otimes A_2 + A_1 \otimes M_2$ 
    where each $M_i$ a maximal ideal in $A_i$, so 
    the result follows from Lemma~\ref{lem:maximal}. 
\end{proof}

\begin{corollary}\label{coro:max-spec-set}
    For any positive integers $m\geq n$, 
    we have that $\fE{n}{m}$ is the ring of regular functions on 
    $\mathcal{V}( \R{n}{m})$. As a set (and as an affine variety), we then have 
    \[
        \maxSpec(\A{n}{m}) 
        = \mathcal{V}(\R{n}{m}) \times \mathbb{A}_{\fk}^{|\F{n}{m}|} 
        = \maxSpec(\fE{n}{m}) \times \maxSpec(\fF{n}{m}). 
    \]
\end{corollary}

\begin{proof}
    Since $\mathcal{V}(\R{n}{m})$ is an affine variety, the first claim follows. 
    The second claim is a direct consequence of Theorem~\ref{thm:Maxideals}. 
\end{proof}

As a consequence, 
we have $\maxSpec(\A{n}{m}) \subseteq \mathbb{A}_{\fk}^{|\G{n}{m}|}$. 
Although in the above corollary we preferred the notation $\mathbb{A}_{\fk}$ over $\fk$ 
due to the topology it suggests, we will not be making this distinction 
in the future. 

\subsection{Maximal ideals via evaluation maps}\label{subsec:max-ideals-evaluation}

In this section we establish a correspondence between certain maximal ideals 
in $\A{n}{m}$ and the maximal ideals of the polynomial ring $\fk[X]$ which, 
geometrically, are supported outside the axes. 

Let $X^{(i)} \coloneqq \{ x_1^{(i)}, x_2^{(i)}, \dots, x_m^{(i)}\}$ and set
$X \coloneqq X^{(1)} \cup \cdots \cup X^{(n)}$. Consider the polynomial ring 
$\fk[X]$ in commuting variables $x_j^{(i)}$ where $1 \leq i \leq n$ and 
$1 \leq j\leq m$. To any tuple ${\underline{a}} = (a_1,\dots,a_k)$ with 
$1\leq a_1 < a_2 < \dots <a_k\leq m$, we assign the monomial $X_{\underline{a}} \coloneqq 
x_{a_1}^{(1)}x_{a_2}^{(2)} \cdots x_{a_k}^{(k)}$.

\begin{theorem}\label{thm:Polynomial-Iso}
    The map defined on tableaux as 
    \begin{equation*}
        \begin{tikzcd}[row sep=0]
            \Omega \colon \A{n}{m} \ar[r] 
            & \fk[X]\\
            \phantom{\Omega \colon} T \ar[r, mapsto] 
            & \displaystyle\prod\limits_{1 \leq i \leq n,\, 1 \leq j \leq m} \left(x_j^{(i)}\right)^{\wt_j^{(i)}(T)}
        \end{tikzcd}
    \end{equation*}
    and extended linearly, is an injective algebra morphism with image 
    \begin{equation*}
        \operatorname{im}(\Omega) =  
        \fk[X_{\underline{a}} \mid \underline{a} \in \mathbb{Z}_{\geq0}^{k},\ 1\leq k \leq n, \; 1\leq a_1 < a_2 < \dots <a_k\leq m].
    \end{equation*}
\end{theorem}

\begin{proof}
    It is easy to check that $\Omega$ is indeed a unital algebra morphism. 
    Injectivity follows from observing that given tableaux $S, T \in \A{n}{m}$ 
    we have $\Omega(S) = \Omega(T)$ if and only if 
    $\wt_j^{(i)}(S) = \wt_j^{(i)}(T)$ 
    for all $1 \leq i \leq n$ and all $1 \leq j \leq m$, which implies $S = T$ 
    because $S, T \in \SSYT_{m}^{n}$. 
    The image is as claimed because given a tuple ${\underline{a}}=(a_1,\dots,a_k)$ 
    with $a_1 < \dots < a_k$ and $1 \leq k \leq n$ we have 
    \begin{equation*}
        \Omega\left(\  
        \hackcenter{
            \tableau{
                a_1\\
                \vdots\vspace{2mm}\\
                a_k}
        }
        \ \right) 
        = X_{\underline{a}} = x_{a_1}^{(1)} \cdots x_{a_k}^{(k)}. 
        \eqno\qed 
    \end{equation*}
    \hideqed
\end{proof}

\begin{remark}\label{rema:compatible-total-orders}
    We can totally order the monomials in $\fk[X]$ lexicographically by first comparing the 
    alphabets and then comparing the variables. As our notation suggests, 
    we will say that 
    $X^{(i)}$ is less than  $X^{(\ell)}$ when $i < \ell$, 
    and for any fixed $i$ we will say that $x_j^{(i)}$ is less than $x_k^{(i)}$ 
    when $j < k$. Explicitly, a monomial 
    $x_{j_1}^{(i_1)} \cdots x_{j_r}^{(i_r)}$ is less than a 
    monomial $x_{k_1}^{(\ell_1)} \cdots x_{k_s}^{(\ell_s)}$, denoted $x_{j_1}^{(i_1)} \cdots x_{j_r}^{(i_r)}<x_{k_1}^{(\ell_1)} \cdots x_{k_s}^{(\ell_s)}$, when either: 
    \begin{enumerate}
        \item $r < s$, or 
        \item $r = s$ and there exists $u$ such that $i_v = \ell_v$ for 
            $1 \leq v \leq u$ with $i_{u+1} < \ell_{u+1}$, or 
        \item $r = s$, $i_u = \ell_u$ for all $u$, and there exists 
            $v$ such that $j_w = k_w$ for $1 \leq w \leq v$ 
            with $j_{v+1} < k_{v+1}$. 
    \end{enumerate}
    In particular, given $T, S \in \SSYT_m^n$ ordered by rows as in 
    Remark~\ref{rem:roworder}, if $T$ is less than or equal to $S$ then 
    $\Omega(T) \leq \Omega(S)$. In this sense, the total ordering on 
    monomials in $\A{n}{m}$ given in Remark~\ref{rem:roworder} 
    is compatible with the total ordering on $\fk[X]$. 
\end{remark}

In what follows, it will be convenient to identify $\fk^{nm}$ with 
$n\times m$ matrices over $\fk$ and index ${\underline{\alpha}}= (\alpha_{1,1},\dots,\alpha_{1,m},\alpha_{2,1},\dots, \alpha_{2,m},\dots,\alpha_{n,m}) \in \fk^{nm}$ 
as a matrix with $(i,j)^{th}$ entry $\alpha_{i,j}$. 

\begin{definition}\label{defi:ordinary-point}
    For any positive integers $m \geq n$, let 
    \begin{align*}
        I(n,m) & \coloneqq \{ {\underline{\alpha}} \in \fk^{nm} \mid \alpha_{i,j} \neq 0\, 
        \text{ for all } 1\leq i \leq n \text{ and all } 1\leq j \leq m \}, 
        \\ 
        J(n,m) & \coloneqq \{ 
        \underline{t} 
        \in \mathcal{V}(\R{n}{m}) \times \fk^{|\F{n}{m}|} \mid 
        t_i \neq 0 \, \text{ for all } i \}. 
    \end{align*}
    We say that a tuple ${\underline{\alpha}} \in \fk^{nm}$ or 
    $\underline{t} \in \mathcal{V}(\R{n}{m}) \times \fk^{|\F{n}{m}|}$ 
    is \newword{ordinary} if 
    ${\underline{\alpha}}  \in I(n,m)$ or $\underline{t} \in J(n,m)$, respectively.
    We call a maximal ideal  $N_{{\underline{\alpha}} }$ in $\fk[X]$ or 
    $M_{\underline{t}}$ in $\A{n}{m}$ \newword{ordinary} if ${{\underline{\alpha}}}$ or 
    $\underline{t}$ are an ordinary tuple, respectively. 
    We let $\mathcal{M}_{ord}(n,m) \coloneqq \{ M_{\underline{t}} \mid \underline{t} \in J(n,m)\}$ 
    be the set of ordinary maximal ideals of $\A{n}{m}$. 
\end{definition}

We have already observed the equality of sets $\fk^{nm} = \maxSpec(\fk[X])$ 
and proven the equality of sets $\mathcal{V}(\R{n}{m}) \times \fk^{|\F{n}{m}|} = \maxSpec(\A{n}{m})$ 
in Corollary~\ref{coro:max-spec-set}. As done in 
Definition~\ref{defi:ordinary-point}, we will explicitly distinguish 
between the tuples ${{\underline{\alpha}}}$ and $\underline{t}$ and the associated maximal 
ideals $N_{{{\underline{\alpha}}}} \in \maxSpec(\fk[X])$ and 
$M_{\underline{t}} \in \maxSpec(\A{n}{m})$ which they index. 

For each ${{\underline{\alpha}}} = (\alpha_{i,j}) \in \fk^{nm}$,
with $1\leq i \leq n$ and $1\leq j \leq m$, 
consider the evaluation homomorphism uniquely defined by the assignment
\[
    \begin{tikzcd}[row sep=0]
        \ev_{{\underline{\alpha}}} \colon \fk[X] \ar[r] 
        & \fk[X]/N_{{{\underline{\alpha}}}} \cong \fk\\
        \phantom{\ev_{{\underline{\alpha}}} \colon} x_j^{(i)} \ar[r, mapsto] 
        & \alpha_{i,j}. 
    \end{tikzcd}
\]
The composition 
$\widetilde{\ev}_{{\underline{\alpha}}} \coloneqq \ev_{{\underline{\alpha}}} \circ \; \Omega$ 
induces an evaluation homomorphism on $\A{n}{m}$ by Theorem~\ref{thm:Polynomial-Iso}. 
Explicitly, it is given on each $T \in \SSYT_m^n$ by 
\[
    \begin{tikzcd}[row sep=0]
        \widetilde{\ev}_{{\underline{\alpha}}} \colon \A{n}{m} \ar[r] 
        & \fk\\
        \phantom{\widetilde{\ev}_{{\underline{\alpha}}} \colon} T \ar[r, mapsto] 
        & \displaystyle\prod\limits_{1 \leq i \leq n,\, 1 \leq j \leq m} 
        \left(\alpha_{i,j}\right)^{\wt_j^{(i)}(T)}
    \end{tikzcd}
\]
and extended linearly. 

Suppose $\G{n}{m}=\{T_1,\dots,T_d\}$ where $d=|\E{n}{m}|+|\F{n}{m}|=\sum_{k=1}^n \binom{m}{k}$. 
Let $\underline{t} \in \mathcal{V}(\R{n}{m}) \times \fk^{|\F{n}{m}|} \subseteq \fk^{|\G{n}{m}|}$, 
and let $M_{\underline{t}}$ be a maximal ideal of $\A{n}{m}$. 
Denote by $\pi_{\underline{t}} : \A{n}{m} \to \A{n}{m}/M_{\underline{t}}$ 
the canonical quotient map. 
\[
    \begin{tikzcd}[row sep=0]
        \pi_{\underline{t}} \colon \A{n}{m} \ar[r] 
        & \A{n}{m}/M_{\underline{t}} \cong \fk\\
        \phantom{\pi_{\underline{t}} \colon} T_i \ar[r, mapsto] 
        & t_i 
    \end{tikzcd}
\]

\begin{proposition}\label{prop:constantterm}
    Suppose $\underline{t}$ is ordinary. 
    Then every nonzero $f \in M_{\underline{t}}$ has a nonzero constant term. 
\end{proposition}

\begin{proof}
   Follows from the fact that a polynomial $f \in \A{n}{m}$ has zero constant term 
    if and only if $\pi_{(0,\dots,0)}(f)=0$. 
\end{proof}

Consider the sets $\Ev(n,m)=\{ \ev_{{\underline{\alpha}}} \mid {{\underline{\alpha}}} \in \fk^{nm}\}$ 
and $\Quo(n,m)=\{ \pi_{\underline{t}}\mid 
\underline{t}\in \mathcal{V}(\R{n}{m}) \times \fk^{|\F{n}{m}|}\}$ with the binary operations 
\begin{equation}\label{eq:binaryoperation}
    \ev_{{\underline{\alpha}}} \circledast \ev_{\underline{\beta}} \coloneqq \ev_{{\underline{\alpha\beta}}} 
    \qquad \text{and} \qquad 
    \pi_{\underline{t}} \circledast \pi_{\underline{s}} \coloneqq \pi_{\underline{t}\underline{s}} 
\end{equation}
where $\underline{\alpha\beta} \coloneqq (\alpha_{i,j}\beta_{i,j})_{1\leq i \leq n, 1\leq j\leq m}$ and 
$\underline{t}\underline{s} \coloneqq (t_is_i)_{1\leq i \leq d}$ 
are the entrywise multiplication of the corresponding tuples. 
It is easy to see that $\Ev(n,m)$ and $\Quo(n,m)$ are monoids but 
not groups under this operation, 
since for any nonordinary ${{\underline{\alpha}}}$ or $\underline{t}$ the corresponding 
map $\ev_{{{\underline{\alpha}}}}$ or $\pi_{\underline{t}}$ will not be invertible 
in $\Ev(n,m)$ or $\Quo(n,m)$, respectively. 

Consider then the following subsets of $\Ev(n,m)$ and $\Quo(n,m)$: 
\begin{align*}
    \Ev_0(n,m) &\coloneqq \{\ev_{{\underline{\alpha}}} \in \Ev(n,m) \mid {{\underline{\alpha}}} \in I(n,m)\}, 
    \\
    \Quo_0(n,m) &\coloneqq \{\pi_{\underline{t}} \in \Quo(n,m) \mid 
    \underline{t} \in J(n,m)\}. 
\end{align*}

Now, $\Ev_0(n,m)$ and $\Quo_0(n,m)$ are commutative groups under the binary operation 
in~\eqref{eq:binaryoperation}, with 
identity elements $\ev_{(1,\dots,1)}$ and $\pi_{(1,\dots,1)}$, respectively. 
The following is an immediate consequence. 

\begin{proposition}\label{prop:maxidealgroup}
    The set $\mathcal{M}_{ord}(n,m)$ is a commutative group under the binary operation 
    $M_{\underline{t}} \circledast M_{\underline{t'}} \coloneqq M_{\underline{tt'}}$ 
    and identity element $M_{(1,\dots,1)}$. 
    In particular, $\mathcal{M}_{ord}(n,m) \cong \Quo_0(n,m)$. 
\end{proposition}

\begin{definition}
    Suppose $\G{n}{m}=\{T_1,\dots,T_d\}$. 
    We define $\Psi : \fk^{nm} \to \fk^{|\G{n}{m}|}$ as follows. 
    \[
        \begin{tikzcd}[row sep=0]
            \Psi \colon \fk^{nm} \ar[r] 
            & \fk^{|\G{n}{m}|}\\
            \phantom{\Psi \colon} {{\underline{\alpha}}} \ar[r, mapsto] 
            & (\widetilde{\ev}_{{\underline{\alpha}}}(T_1), \dots, \widetilde{\ev}_{{\underline{\alpha}}}(T_d))
        \end{tikzcd}
    \]
\end{definition}

\begin{example}\label{exam:alpha-psi-2-T-3}
    Let $\alpha_{i, a}=a^i$ which induces the map 
    $\widetilde{\ev}_{{\underline{\alpha}}}: \A{2}{3}\rightarrow \fk$ defined on $\G{2}{3}$ as follows. 
    \[
        \hackcenter{\tableau{1}} \mapsto 1 \quad 
        \hackcenter{\tableau{2}} \mapsto 2 \quad     
        \tableau{1\\3} \mapsto 1 \cdot 3^2=9 \quad 
        \tableau{2\\3} \mapsto 2 \cdot 3^2=18\quad
      \hackcenter{\tableau{3}} \mapsto 3 \quad 
        \tableau{1\\2} \mapsto 1 \cdot 2^2=4 
    \]
    This coincides with the projection 
    $\pi_{\Psi({{\underline{\alpha}}})} : \A{2}{3}/M_{(1,2,9,18,3,4)} \to \fk$ 
    for the maximal ideal 
    \[
        M_{(1,2,9,18,3,4)} = \left\langle 
        \hackcenter{\tableau{1}}-1, 
        \hackcenter{\tableau{2}}-2, 
        \hackcenter{\tableau{1\\3}}-9, 
        \hackcenter{\tableau{2\\3}}-18,
        \hackcenter{\tableau{3}}-3, 
        \hackcenter{\tableau{1\\2}}-4 
        \right\rangle. 
    \]
    Moreover, note that there are entries of ${{\underline{\alpha}}}$ that are inconsequential. 
    For example, the value of $\alpha_{2, 1}$ is never used to compute $\Psi({{\underline{\alpha}}})$ 
    because $x_1^{(2)}$ does not divide $\Omega(T)$ for any $T \in \SSYT_3^2$. 
\end{example}

\begin{remark}\label{rem:homomorphism}
    Observe that the maps $\ev_{{\underline{\alpha}}}$, $\Omega$, and 
    $\pi_{\underline{t}}$ are linear and multiplicative on tableaux. Namely 
    let $a\in \fk$ be a scalar, ${{\underline{\alpha}}} \in \fk^{nm}$ and 
    $\underline{t} \in \mathcal{V}(\R{n}{m})\times\fk^{|\F{n}{m}|}$ be tuples, 
    and $T,S \in \SSYT_m^n$, then the following equalities hold: 
    \begin{align*}
        \widetilde{\ev}_{{\underline{\alpha}}}(aT+S) &=a\widetilde{\ev}_{{\underline{\alpha}}}(T) +\widetilde{\ev}_{{\underline{\alpha}}}(S)&
        \widetilde{\ev}_{{\underline{\alpha}}}(T\star S)&=\widetilde{\ev}_{{\underline{\alpha}}}(T)\widetilde{\ev}_{{\underline{\alpha}}}(S)\\
        \pi_{\underline{t}}(aT+S)&=a\pi_{\underline{t}}(T)+\pi_{\underline{t}}(S) &
        \pi_{\underline{t}}(T\star S)&=\pi_{\underline{t}}(T)\pi_{\underline{t}}(S).
    \end{align*}
    However, note that for ${\underline{\beta}} \in \fk^{nm}$ and 
    $\underline{s} \in \mathcal{V}(\R{n}{m})\times\fk^{|\F{n}{m}|}$  
    as well as an 
    arbitrary element $f \in \A{n}{m}$, it is not true that the following 
    equalities hold: 
    \begin{align}\label{eq:falsehom}
        \widetilde{\ev}_{\underline{
        \alpha\beta}}(f) = \widetilde{\ev}_{{{\underline{\alpha}}}}(f)\widetilde{\ev}_{{\underline{\beta}}}(f) 
        & & 
        \pi_{\underline{ts}}(f)=\pi_{\underline{t}}(f)\pi_{\underline{s}}(f). 
    \end{align}
    In the particularly special case when $T$ is a column tableau, 
    the equalities~\eqref{eq:falsehom} do hold. This observation will be 
    key in justifying why the map $\Psi_0$ below is a group homomorphism. 
\end{remark}

\begin{lemma}\label{lem:evaluation-alpha-omega-is-projection-omega-alpha}
    Let ${{\underline{\alpha}}} \in \fk^{nm}$, then 
    $\Psi({{\underline{\alpha}}}) \in \mathcal{V}(\R{n}{m}) \times \fk^{|\F{n}{m}|}$ 
    and $\widetilde{\ev}_{{\underline{\alpha}}} = \pi_{\Psi({\underline{\alpha}})}$. 
    In particular, $\ev_{{\underline{\alpha}}}(\Omega(f)) = \pi_{\Psi({\underline{\alpha}})}(f)$ 
    for all $f \in \A{n}{m}$. 
\end{lemma}

\begin{proof}
    Let 
    $\Psi({\underline{\alpha}}) = (t_1,\dots,t_{|\E{n}{m}|},t'_1,\dots,t'_{|\F{n}{m}|}) 
    \in \fk^{|\E{n}{m}|} \times \fk^{|\F{n}{m}|}$, 
    $\G{n}{m} = \{ T_1, \dots, T_d\}$, and suppose 
    $T_i \star T_j - T_k \star T_\ell \in \R{n}{m}$ for some distinct 
    $1\leq i,j,k,\ell \leq d$. Then, by Remark~\ref{rem:homomorphism}
    \[t_it_j 
    = \widetilde{\ev}_{\underline{\alpha}}(T_i)\widetilde{\ev}_{\underline{\alpha}}(T_j) 
    = \widetilde{\ev}_{\underline{\alpha}}(T_i\star T_j) 
    = \widetilde{\ev}_{\underline{\alpha}}(T_k\star T_\ell) 
    = \widetilde{\ev}_{\underline{\alpha}}(T_k)\widetilde{\ev}_{\underline{\alpha}}(T_\ell) 
    = t_kt_\ell,\] so 
    $(t_1,\dots,t_{|\E{n}{m}|}) \in \mathcal{V}(\R{n}{m})$ 
    and the first claim holds. 
    To finish the proof, it suffices to check the last equality on the 
    generators. 
    Since for all $1 \leq i \leq d$ we indeed have 
    $\pi_{\Psi({\underline{\alpha}})}(T_i) = \Psi({\underline{\alpha}})_i 
    = \widetilde{\ev}_{{\underline{\alpha}}}(T_i) = \ev_{\underline{\alpha}}(\Omega(T_i))$, and 
    the result follows. 
\end{proof}

The following establishes an equivalence between the ordinary maximal ideals 
in the tableaux algebra $\A{n}{m}$ and those in the polynomial ring $\fk[X]$. 

\begin{theorem}\label{thm:EvalsMaxIdeal}
    The map $\Psi_0: \Ev_0(n,m) \to \Quo_0(n,m)$ defined by 
    $\Psi_0(\ev_{\underline{\alpha}}) \coloneqq \pi_{\Psi({\underline{\alpha}})}$ is a surjective 
    group homomorphism with kernel 
    \[
        \ker\Psi_0 = \{ \ev_{\underline{\alpha}} |\; {\underline{\alpha}} \in I(n,m) \text{ and } 
        \alpha_{i,j} = 1 \text{ for all }  i\leq j\}. 
    \]
    Consequently 
    $\mathcal{M}_{ord}(n,m) \cong \Ev_0(n,m)/\ker\Psi_0$. 
\end{theorem}

\begin{proof}
    Clearly $\Psi_0$ is well defined, because $\ev_{\underline{\alpha}} \in \Ev_0(n,m)$ 
    implies ${\underline{\alpha}} \in I(n,m)$, so $\Psi({\underline{\alpha}}) \in J(n,m)$ and hence 
    $\pi_{\Psi({\underline{\alpha}})} \in \Quo_0(n,m)$. 
    Given ${\underline{\alpha}}, {\underline{\beta}} \in \fk^{nm}$ and $T \in \G{n}{m}$, then 
    $\widetilde{\ev}_{{\underline{\alpha\beta}} }(T) 
    = \widetilde{\ev}_{{\underline{\alpha}}}(T)\widetilde{\ev}_{{\underline{\beta}}}(T)$ 
    by Remark~\ref{rem:homomorphism}. Thus, 
    $\Psi({\underline{\alpha\beta}}) 
    = \Psi({\underline{\alpha}})\Psi({\underline{\beta}})$ and 
    \[\Psi_0(\ev_{{\underline{\alpha}}} \circledast \ev_{{\underline{\beta}}}) 
    = \Psi_0(\ev_{\underline{\alpha\beta}}) 
    = \pi_{\Psi({\underline{\alpha\beta}})} 
    = \pi_{\Psi({\underline{\alpha}})} \circledast \pi_{\Psi({\underline{\beta}})} 
    = \Psi_0(\ev_{{\underline{\alpha}}}) \circledast \Psi_0(\ev_{{\underline{\beta}}}),\] 
    so $\Psi_0$ is a group homomorphism. 
    Note that $\ev_{\underline{\alpha}} \in \ker \Psi_0$ if and only if 
    $\widetilde{\ev}_{\underline{\alpha}}(T) = 1$ for all $T \in \G{n}{m}$, which holds 
    if and only if $\alpha_{i,j}=1$ for all $ i\leq j$, so 
    $\ker\Psi_0$ is as claimed. Assuming $\Psi_0$ is surjective, 
    the final claim follows by the first 
    isomorphism theorem and Proposition~\ref{prop:maxidealgroup}. 

    Lastly, we show $\Psi_0$ is surjective. 
    Given $\underline{t} = (t_1,\dots,t_{|\G{n}{m}|}) 
    \in \mathcal{V}(\R{n}{m}) \times \fk^{|\F{n}{m}|}$ 
    an ordinary point, it suffices to find 
    ${\underline{\beta}} \in \Psi^{-1}(\underline{t})$. Recall that for 
    any column $T \in \G{n}{m}$ of height $h$ 
    we can associate a monomial $\Omega(T)=X_{\underline{a}} 
    = x_{a_1}^{(1)}\dots x_{a_h}^{(h)}$ such that $a_1 < \dots < a_h$ and 
    $i \leq a_i$ for all $1 \leq i \leq h$, 
    as in the proof of Theorem~\ref{thm:Polynomial-Iso}. 
    We will construct ${\underline{\beta}}$ 
    by inducting on the 
    height $h$ of the columns in $\G{n}{m}$ as follows. 
    
    Suppose $T_k \in \G{n}{m}$ is a column of height $h=1$, 
    so $\Omega(T_k) = x_{a_1}^{(1)}$ for some $1\leq a_1\leq m$, 
    and set $\beta_{1,a_1} \coloneqq t_k$. 
    These assignments determine $\beta_{1,j}$ for all $1\leq j \leq m$. 
    Suppose $\beta_{i,j}$ has been determined for all 
    $1\leq i <h$ and $i\leq j$ and that $T_k$ is a column of height 
    $h>1$, so $\Omega(T_k)=X_{\underline{a}}=x_{a_1}^{(1)}\dots x_{a_h}^{(h)}$, 
    and set 
    \[
        \beta_{h,a_h} \coloneqq \frac{t_k}{\prod_{1\leq i<h} \beta_{i,a_i}}. 
    \]
    Note $\prod_{1\leq i<h} \beta_{i,a_i}\neq 0$ because
    $\underline{t}$ is ordinary. 
    Moreover, if $T_k$ and $T_\ell$ are two tableaux of height $h$ 
    with the same last entry, 
    say $\Omega(T_k)=X_{\underline{a}}$ and $\Omega(T_\ell)=X_{\underline{b}}$, 
    then
    \[
        t_k \left( \prod_{1\leq i<h} \beta_{i,b_i}\right) 
        = \pi_{\underline{t}}(T_k \star T_\ell') 
        = \pi_{\underline{t}}(T'_k \star T_\ell) 
        = \left( \prod_{1\leq i<h} \beta_{i,a_i}\right) t_\ell 
    \]
    where $T_k'$ and $T'_\ell$ are obtained by deleting the last entry of 
    $T_k$ and $T_\ell$, respectively. 
    This procedure is thus well defined, and 
    $\widetilde{\ev}_{\underline{\beta}}(T_k) = \beta_{1,a_1} \cdots \beta_{h,a_h} = t_k$, 
    so $\Psi({\underline{\beta}}) = \underline{t}$ as desired. 
\end{proof}

In particular, observe that as a consequence of the above proof, when 
$\underline{t}$ is ordinary the preimage $\Psi^{-1}(\underline{t})$ 
has dimension $\binom{n}{2}$. 

As noted above, the nonordinary points of 
$\mathcal{V}(\R{n}{m}) \times \fk^{|\F{n}{m}|}$ 
coincide with the noninvertible elements in $\Quo(n,m)$. Therefore, in order to 
study these points, we use that $\Ev(n,m)$ and $\Quo(n,m)$ form commutative 
monoids under pointwise multiplication. 

\begin{definition}
    Let $\Psi_{*}: \Ev(n,m) \to \Quo(n,m)$ to be the monoid morphism defined 
    by setting $\Psi_{*}(\ev_{\underline{\alpha}}) \coloneqq \pi_{\Psi({\underline{\alpha}})}$ 
    for any ${\underline{\alpha}} \in \fk^{nm}$. 
\end{definition}

The map is well defined by 
Lemma~\ref{lem:evaluation-alpha-omega-is-projection-omega-alpha} 
and it is a monoid morphism by the proof of Theorem~\ref{thm:EvalsMaxIdeal}. 
Its easy to see $\Psi_{*}\vert_{\Ev_0(n,m)} = \Psi_0$, 
hence $\Psi_{*}$ extends the group morphism $\Psi_0$ with source 
$\Ev_0(n,m)$ to a monoid morphism with source $\Ev(n,m)$. 

\begin{theorem}
    The monoid morphism $\Psi_{*}: \Ev(n,m) \to \Quo(n,m)$ 
    fits in the diagram 
    \[
        \begin{tikzcd}[row sep=0]
            1 \ar[r] 
            & \ker\Psi_{*} \ar[r, hook] 
            & \Ev(n,m) \ar[r, "\Psi_{*}"] 
            & \Quo(n,m) \ar[r, two heads] 
            & \Quo(n,m) / \operatorname{im}\Psi_{*} \ar[r] 
            & 1 
        \end{tikzcd}
    \]
    where the kernel and image of $\Psi_{*}$ are given as follows. 
    \begin{align*}
        \ker\Psi_{*}=&
         \{ \ev_{\underline{\alpha}} \mid \alpha_{i,j}=1 \text{ for all } i\leq j\}
        \\ 
        \operatorname{im}\Psi_{*}=&
        \Quo_0(n,m)
        \cup 
        \{
            \pi_{\underline{t}} \mid \text{if } t_i=0 \text{ and } 
            \Omega(T_i) \text{ divides } \Omega(T_k) \text{ then } t_k=0 
        \}
    \end{align*}
    In particular, the equivalence relation given by $\operatorname{im}\Psi_{*}$ 
    is a congruence, so the corresponding quotient 
    $\Quo(n,m) / \operatorname{im}\Psi_{*}$ is 
    a monoid in the natural way. 
\end{theorem}

\begin{proof}
    We know $\{ \ev_{\underline{\alpha}} \mid \alpha_{i,j} = 
    1 \text{ for all }  i\leq j\} \subseteq \ker\Psi_{*}$ 
    by the proof of Theorem \ref{thm:EvalsMaxIdeal}. 
    Since the entries $\alpha_{i,j}$  with $j<i$ do not play a role 
    in the definition of $\Psi({\underline{\alpha}})$, the first equality follows. 

    Suppose ${\underline{\alpha}} \in I(n,m)$, since $\Psi_0$ is surjective, 
    we have $\operatorname{im}\Psi_0 = \Quo_0(n,m) \subseteq 
    \operatorname{im}\Psi_{*}$. 
    Suppose ${\underline{\alpha}} \notin I(n,m)$, so that 
    $\alpha_{i,j}=0$ for some $i$ and $j$. Then for any $T_k \in \G{n}{m}$ 
    for which $x^{(i)}_j$ divides $\Omega(T_k)$, we must have $t_k=0$. 
    Thus $\pi_{\Psi({\underline{\alpha}})}$ satisfies that 
    given pair of columns $T_\ell,T_k \in \G{n}{m}$ with 
    $\Omega(T_\ell)$ dividing $\Omega(T_k)$, 
    if $t_\ell = 0$ then $t_k=0$, 
    as desired for the second equality. 

    Let $\pi_{\underline{r}}, \pi_{\underline{s}}, \pi_{\underline{u}}, 
    \pi_{\underline{v}}\in \Quo(n,m)$ such that $\pi_{\underline{r}} 
    \equiv \pi_{\underline{s}}$ 
    and $\pi_{\underline{u}} \equiv \pi_{\underline{v}}$ in 
    $\Quo(n,m) / \operatorname{im}\Psi_{*}$. 
    That is, there exist $\pi_{\underline{t}}, 
    \pi_{\underline{w}} \in \operatorname{im}\Psi_{*}$ such that 
    $\pi_{\underline{s}} = \pi_{\underline{r}} \circledast \pi_{\underline{t}}$ 
    and 
    $\pi_{\underline{v}} = \pi_{\underline{u}} \circledast \pi_{\underline{w}}$ 
    in $\Quo(n,m)$. 
    We have that $\Quo(n,m)$ and $\operatorname{im}\Psi_{*}$ are commutative 
    monoids as a consequence of 
    Remark~\ref{rem:homomorphism}, so 
    $\pi_{\underline{s}} \circledast \pi_{\underline{v}} 
    = \pi_{\underline{r}} \circledast \pi_{\underline{t}} 
    \circledast \pi_{\underline{u}} \circledast \pi_{\underline{w}} 
    = \pi_{\underline{r}} \circledast \pi_{\underline{u}} 
    \circledast \pi_{\underline{t}} \circledast \pi_{\underline{w}}$ 
    in $\Quo(n,m)$ with $\pi_{\underline{t}} \circledast \pi_{\underline{w}} 
    \in \operatorname{im}\Psi_{*}$. 
    Thus $\pi_{\underline{s}} \circledast \pi_{\underline{v}} 
    \equiv \pi_{\underline{r}} \circledast \pi_{\underline{u}}$ in 
    $\Quo(n,m) / \operatorname{im}\Psi_{*}$ and 
    the remaining claim follows. 
\end{proof}

\begin{remark}
Note that in the above diagram, the image of a map coincides with the kernel 
of the following one. While some authors may call this 
an ``exact sequence of monoids'', 
we will avoid such a slippery notion~\cite[Remark 2.6]{balmer2005quadratic} 
and we will not enter into the technical details of such a statement here. \end{remark}

\begin{example}
    Consider again $\A{2}{3}$ and $\alpha_{i, a}=a^i$ 
    as in Example~\ref{exam:alpha-psi-2-T-3}. 
    We have the single product relation~\eqref{eq:first-product-relation}, 
    which we 
    now use to showcase how $\Psi_{*} : \Ev(2,3) \to \Quo(2,3)$ is not surjective. 
    Let $\underline{t} = (0,0,9,18,0,4)$, 
    giving $\pi_{\underline{t}}\in \Quo(2,3)$ 
    whose value on $\G{2}{3}$ follows. 
    \[
        \hackcenter{\tableau{1}} \mapsto 0 \qquad 
        \hackcenter{\tableau{2}} \mapsto 0  \qquad 
        \tableau{1\\3} \mapsto 9 \qquad 
        \tableau{2\\3} \mapsto 18 \qquad 
        \hackcenter{\tableau{3}} \mapsto 0 \qquad 
        \tableau{1\\2} \mapsto 4
    \]
    This indexes a maximal ideal 
    $M_{\underline{t}} \in \maxSpec(\A{2}{3})$ because the 
    following product relation holds. 
    \[
        \pi_{\underline{t}} \left( \; \hackcenter{\tableau{1\\3}} \; \right) 
        * \pi_{\underline{t}} \left( \;\hackcenter{\tableau{2}}\;\right) 
        = 9 \cdot 0 = 18 \cdot 0 
        = \pi_{\underline{t}} \left( \;\hackcenter{\tableau{2\\3}}\;\right) 
        * \pi_{\underline{t}} \left( \;\hackcenter{\tableau{1}}\;\right) 
    \]
    Although $\pi_{\underline{t}}$ coincides with 
    $\ev_{\Psi({\underline{\alpha}})}$ in the long columns, they 
    differ in the short columns. 
    Suppose there is ${\underline{\beta}} \in \fk^{6}$ such that 
    $\Psi_{*}({\underline{\beta}}) = \pi_{\underline{t}}$. Then 
    \[
        9 = \pi_{\underline{t}} \left(\; \hackcenter{\tableau{1\\3}} \; \right) 
        = \beta_{1,1} \cdot \beta_{2,3} = \pi_{\underline{t}} 
        \left( \;\hackcenter{\tableau{1}}\;\right) \cdot \beta_{2,3} 
        = 0 \cdot \beta_{2,3} = 0, 
    \]
    a contradiction. 
    Thus, $\pi_{\underline{t}} \notin \operatorname{im}\Psi_{*}$. 
\end{example}

\subsection{The Zariski topology of the maximal spectrum}\label{sec:topology-max-spec}

In this section we build on the previous one to describe the topology of 
$\maxSpec(\A{n}{m})$ in terms of the topology of $\maxSpec(\fk[X])$. 

Recall that a basis of opens for the topology of $\maxSpec(\A{n}{m})$ is given by 
$\Delta(\maxSpec(\A{n}{m})) \coloneqq \{ D(f) \mid f \in \A{n}{m}\}$ where 
$D(f) \coloneqq \{M \in \maxSpec(\A{n}{m}) \mid f \notin M\}$ for $f \in \A{n}{m}$. 
A similarly defined $\Delta(\maxSpec(\fk[X]))$ yields a basis of opens for 
$\maxSpec(\fk[X])$. 
Moreover, by Theorem~\ref{thm:Maxideals} points in $\maxSpec(\A{n}{m})$ can be indexed by $M_{\underline{t}}$ 
with $\underline{t} \in \mathcal{V}(\R{n}{m}) \times \fk^{|\F{n}{m}|}$, 
which in turn identifies $D(f)$ with the set $\{\underline{t} \in \mathcal{V}(\R{n}{m}) 
\times \fk^{|\F{n}{m}|} \mid f \notin M_{\underline{t}}\}$. 
Similarly, 
points in $\maxSpec(\fk[X])$ can be indexed by $N_{{\underline{\alpha}}}$ 
with ${\underline{\alpha}} \in \fk^{nm}$, and its basic opens enjoy an analogous identification. 

\begin{lemma}\label{lemm:open-basis-max-spec}
    The assignment 
    \begin{equation*}
        \label{eq:basis-max-spec}
        \begin{tikzcd}[row sep=0]
            \Xi \colon \Delta(\maxSpec(\A{n}{m})) \ar[r] 
            & \Delta(\maxSpec(\fk[X]))\\
            \phantom{\Xi \colon} D(f) \ar[r, mapsto] 
            & \{ N_{{\underline{\alpha}}} \in \maxSpec(\fk[X]) \mid 
            \exists M_{\underline{t}} \in D(f) \text{ with } \Psi({\underline{\alpha}}) = \underline{t} \}
        \end{tikzcd}
    \end{equation*}
    is well defined. Moreover $\Xi(D(f)) = D(\Omega(f))$. 
\end{lemma}

\begin{proof}
    It suffices to prove the second claim. Let $f \in \A{n}{m}$, note 
    \begin{align*}
        D(\Omega(f)) 
        &= \{N_{{\underline{\alpha}}} \in \maxSpec(\fk[X]) \mid \Omega(f) \notin N_{{\underline{\alpha}}}\}\\ 
        &= \{N_{{\underline{\alpha}}} \in \maxSpec(\fk[X]) \mid \ev_{{\underline{\alpha}}}(\Omega(f)) \neq 0\}\\ 
        &= \{N_{{\underline{\alpha}}} \in \maxSpec(\fk[X]) \mid \pi_{\Psi({\underline{\alpha}})}(f) \neq 0\}\\ 
        &= \{N_{{\underline{\alpha}}} \in \maxSpec(\fk[X]) \mid 
        \exists M_{\Psi({\underline{\alpha}})} \in \maxSpec(\A{n}{m}) \text{ with } f \notin M_{\Psi({\underline{\alpha}})}\}\\ 
        &= \{N_{{\underline{\alpha}}} \in \maxSpec(\fk[X]) \mid 
        \exists M_{\underline{t}} \in \maxSpec(\A{n}{m}) \text{ with } 
        f \notin M_{\underline{t}} \text{ and } \Psi({\underline{\alpha}}) = \underline{t}\}\\ 
        &= \{N_{{\underline{\alpha}}} \in \maxSpec(\fk[X]) \mid 
        \exists M_{\underline{t}} \in D(f) \text{ with } \Psi({\underline{\alpha}}) = \underline{t}\}\\ 
        &= \Xi(D(f)). 
    \end{align*}
    Hence $\Xi(D(f)) = D(\Omega(f))$ and $\Xi$ is well defined. 
\end{proof}

As corollary, we obtain the desired description of the topology. 

\begin{theorem}\label{thm:opens-max-spec}
    The assignment 
    \begin{equation*}
        \label{eq:topology-max-spec}
        \begin{tikzcd}[row sep=0]
            \Theta \colon \mathrm{Open}(\maxSpec(\A{n}{m})) \ar[r] 
            & \mathrm{Open}(\maxSpec(\fk[X]))\\
            \phantom{\Theta \colon} U \ar[r, mapsto] 
            & \{ N_{{\underline{\alpha}}} \in \maxSpec(\fk[X]) \mid 
            \exists M_{\underline{t}} \in U \text{ with } \Psi({\underline{\alpha}}) = \underline{t} \}
        \end{tikzcd}
    \end{equation*}
    is well defined. Moreover $\Theta(D(f)) = \Xi(D(f))$, 
    namely $\Theta$ restricts to $\Xi$. 
\end{theorem}

\begin{proof}
    Indeed $\Theta(D(f)) = \Xi(D(f))$ by definition. Let $U$ be an open of 
    $\maxSpec(\A{n}{m})$, say $U = \bigcup_{i \in I}D(f_{i})$ for a set 
    $\{f_{i} \in \A{n}{m}\}_{i \in I}$. Then 
    \begin{align*}
        \Theta(U) 
        &= \Theta\left(\bigcup_{i \in I}D(f_{i})\right)\\
        &= \{N_{{\underline{\alpha}}} \in \maxSpec(\fk[X]) \mid 
        \exists M_{\underline{t}} \in \bigcup_{i \in I}D(f_{i}) \text{ with } 
        \Psi({\underline{\alpha}}) = \underline{t}\}\\ 
        &= \{N_{{\underline{\alpha}}} \in \maxSpec(\fk[X]) \mid 
        \exists i \in I \text{ with } M_{\underline{t}} \in D(f_{i}) \text{ and } 
        \Psi({\underline{\alpha}}) = \underline{t}\}\\ 
        &= \bigcup_{i \in I}\{N_{{\underline{\alpha}}} \in \maxSpec(\fk[X]) \mid
        \exists M_{\underline{t}} \in D(f_{i}) \text{ with } 
        \Psi({\underline{\alpha}}) = \underline{t}\}\\ 
        &= \bigcup_{i \in I}\Theta(D(f_{i})) = \bigcup_{i \in I}\Xi(D(f_{i})) 
        = \bigcup_{i \in I}D(\Omega(f_{i})). 
    \end{align*}
    Hence, $\Theta(U)$ is an open of $\maxSpec(\fk[X])$. 
\end{proof}

\begin{corollary}\label{cor:continuous-map-max-spec}
    The map $\Psi^*$ given below is continuous,
    \begin{equation*} 
        \begin{tikzcd}[row sep=0]
            \Psi^* \colon \maxSpec(\fk[X]) \ar[r] 
            & \maxSpec(\A{n}{m})\\
            \phantom{\Psi^* \colon} N_{{\underline{\alpha}}} \ar[r, mapsto] 
            & M_{\Psi({\underline{\alpha}})}. 
        \end{tikzcd}
    \end{equation*} 
\end{corollary}

\begin{proof}
    Follows directly from Theorem~\ref{thm:opens-max-spec}. 
\end{proof}

\subsection{Prime ideals}\label{sec:prime-ideals-tableau-algebra} 

In this section we study a small 
class of interesting ideals coming from the representation theory of $\mathfrak{sl}_m$, 
and discuss which of these are prime. 
Finding a complete classification of 
the prime ideals of $\A{n}{m}$ is beyond the scope of this paper.

We begin by first studying the prime principal ideals of the form 
$\langle T - a \rangle_{\A{n}{m}}$ for any monomial $T \in \SSYT_m(\lambda)$ 
and $a \in \fk$. This turns out to be intimately related to 
the decomposition of $\A{n}{m}$ given in Corollary \ref{cor:tensorprod}. 

\begin{lemma}\label{lem.multi_col}
    Suppose $\lambda=(\lambda_1,\dots,\lambda_n)$ with $\lambda_1 \geq 2$ 
    and $T \in \SSYT_m(\lambda)$. 
    Then, the ideal $\langle T \rangle_{\A{n}{m}}$ is not prime. 
\end{lemma}

\begin{proof}
    Since $\lambda_1 \geq 2$, then any $T \in \SSYT_m(\lambda)$ has at least two columns and can be 
    decomposed as $T = C\star T'$ for some $C \in \G{n}{m}$ and $T' \in \SSYT_m^n$ 
    satisfying $\operatorname{shape}(C) + \operatorname{shape}(T') = \lambda$. 
    In particular, any tableau $S \in \langle T\rangle$ must have 
    $\operatorname{shape}(S) \geq \lambda$ in dominance order, so that
    neither $C$ not $T'$ can be in $\langle T\rangle$, finishing the proof. 
\end{proof}

\begin{proposition}
    Given any nonempty $T \in \SSYT_m^n$, the principal ideal 
    $\langle T \rangle_{\A{n}{m}}$ is prime in $\A{n}{m}$ 
    if and only if $T \in \F{n}{m}$. 
\end{proposition}

\begin{proof}
    This follows directly from Lemma~\ref{lem.multi_col} and 
    Lemma~\ref{lem:prod-ideal}. 
\end{proof}

More generally, when $a \neq 0$, all such ideals with $T$ any column are prime.

\begin{theorem}
    For any column tableau $T$ and $a \neq 0$, 
    the ideal $\langle T - a \rangle_{\A{n}{m}}$ is prime. 
\end{theorem}

\begin{proof}
    Suppose $f, g \in \A{n}{m}$ with $f \star g \in \langle T - a \rangle_{\A{n}{m}}$. 
    Let $\widetilde{f}$ and $\widetilde{g}$ be the representatives of 
    $f$ and $g$ in $\fk[\G{n}{m}]$ satisfying that if $T$ divides a monomial 
    of $f$ or $g$ then $T$ appears as a factor in the corresponding monomial 
    of $\widetilde{f}$ or $\widetilde{g}$, respectively. 
    Then $\widetilde{f} \cdot \widetilde{g} \in \fk[\G{n}{m}]$ is a representative of 
    $f \star g \in \A{n}{m}$, and $f \star g \in \langle T - a \rangle_{\A{n}{m}}$ 
    implies that $\widetilde{f} \cdot \widetilde{g}$ has a zero at $T = a$. 
    Since $\fk[\G{n}{m}]$ is an integral domain, without loss of 
    generality we may assume that $\widetilde{f}$ has a zero at $T = a$. 
    Consider now $\widetilde{f} \in (\fk[\G{n}{m} \setminus \{ T \}])[T]$ as a 
    polynomial in the single variable $T$ with coefficients in 
    $\fk[\G{n}{m} \setminus \{ T \}]$. Then $\widetilde{f}$ is divisible by 
    $T - a$, so it can be factored as $\widetilde{f} = (T-a) \widetilde{h}$ 
    for some $\widetilde{h} \in \fk[\G{n}{m} \setminus \{ T \}]$. 
    This implies $f \in \langle T - a \rangle_{\A{n}{m}}$ and concludes the proof. 
\end{proof}

For a fixed $m>0$, the finite dimensional highest weight irreducible 
representations $V(\lambda)$ of $\mathfrak{sl}_m$ are indexed by partitions $\lambda$ 
with at most $m$ rows, with \newword{crystal} or \newword{canonical} basis 
indexed by the set $\SSYT_m(\lambda)$ (see Section \ref{sec:crystals}). 
Indeed, if $\lambda$ has at most $n$ parts for some $n\leq m$, 
then  the crystal basis of the restriction $\operatorname{Res}_m^n(V(\lambda))$ 
as a module over $\mathfrak{sl}_n$ is in bijection with the set $\SSYT_m(\lambda)$. 
This motivates the following definition. 

\begin{definition}
    For $\lambda$ a partition having at most $n \leq m$ parts, 
    we define the \newword{crystal ideal} 
    $I(\lambda)$ to be the ideal in $\A{n}{m}$ generated by the set $\SSYT_m(\lambda)$. 
\end{definition}

\begin{lemma}\label{lem.shape}
    An element $f \in \A{n}{m}$ is in $I(1^k)$ if and only if every summand 
    of $f$ contains a column of height exactly $k$. 
\end{lemma}

\begin{proof}
    Suppose $f=\sum_{i} a_iT_i$ for some $a_i \in \fk$ and $T_i \in \SSYT_m^n$. 
    If every $T_i$ contains a column of height exactly $k$ then 
    $T_i = C_i \star S_i$ for some $C_i \in I(1^k)$ and $S_i \in \SSYT_m^n$, 
    so $f \in I(1^k)$. 
    If $f\in I(1^k)$ then $f=\sum_{j} b_j C_j\star T'_j$ for some 
    $b_j \in \fk$, $C_j \in I(1^k)$, and $T'_j \in \SSYT_m^n$, 
    so every summand of $f$ is in $I(1^k)$. 
\end{proof}

With this in hand we give a complete characterization of the prime crystal ideals of $\A{n}{m}$.

\begin{theorem}
    $I(\lambda)$ is prime in $\A{n}{m}$ if and only if 
    $\lambda = (1^k)$ for some $1 \leq k \leq n$. 
\end{theorem}

\begin{proof}
    The forward direction follows directly from Lemma~\ref{lem.multi_col}. 
    For the backward direction, suppose $\lambda = (1^k)$ for some 
    $1\leq k \leq n$. If $I(1^k)$ is not prime then there exists 
    $f \in I(1^k)$ with $f = gh$ and $g,h \notin I(1^k)$. 
    Write $g=g'+\sum_i a_i T_i$ and 
    $h = h'+\sum_j b_jS_j$ for some nonzero $a_i,b_j \in \fk$ such that 
    $g', h' \in I(1^k)$ and $T_i, S_j \notin I(1^k)$. Then $T_i$ and $S_j$ 
    do not contain a column of height $k$ for all $i$ and $j$ by 
    Lemma~\ref{lem.shape}, so $T_iS_j$ does not contain a column of height $k$ 
    for all $i$ and $j$, so $(\sum_i a_iT_i)(\sum_j b_j S_j) \notin I(1^k)$. 
    Also $(\sum_i a_iT_i)(\sum_j b_j S_j) = 
    gh - g'h - h'g + g'h' \in I(1^k)$, a contradiction. 
\end{proof}

Moreover, as a corollary of Theorem~\ref{prop:IntegralDomain-Noetherian-Reduced-Jacobson} 
we obtain that all the prime ideals above arise as the intersection 
of the maximal ideals that contain them. 

\begin{corollary}\label{thm:crystal-intersection}
    Let $\G{n}{m} = \{ T_1, \dots, T_d \}$ and $a \neq 0$. 
    \begin{enumerate}
        \item $I(1^k) = \bigcap_{\underline{t}} M_{\underline{t}}$ such that 
            $\pi_{\underline{t}}(T)=0$ for all $T \in \SSYT_m^n(1^k)$. 
        \item $\langle T_{i} - a \rangle = \bigcap_{\underline{t}} M_{\underline{t}}$ 
            such that $t_i = a$. 
    \end{enumerate}
\end{corollary}

%%%%%%%%%%%%%%%%%%%%%%%%%%%%%%%%%%%%%%%%%%%%%%%%%%%%%%%%%%%%%%%%%%%%%%%%%%%%%%%%

\section{Representations}\label{sec:rep}

In this section we briefly justify the difficulty in understanding the 
representation theory of the tableaux algebra, 
showing that it contains the representation theory 
of all finite dimensional algebras. 
Along the way we give a complete description of all the finite dimensional 
irreducible representations of the tableaux algebra. 

It is well known that the Drozd ring $\fk[x,y]/(x^2,xy^2,y^3)$ 
is of wild representation type~\cite{drozd:wild}. 
This means that there is a representation embedding 
$\rep(\fk\langle x_1,x_2\rangle) \rightarrowtail \rep(\fk[x,y]/(x^2,xy^2,y^3))$, 
or equivalently that for any finitely generated algebra 
$\Lambda$ there exists a representation embedding 
$\rep(\Lambda) \rightarrowtail \rep(\fk[x,y]/(x^2,xy^2,y^3))$, 
see for example~\cite{ringel:rep-type}. 
Since there is a surjective map of algebras 
$\A{n}{m} \twoheadrightarrow \fk[x,y]/(x^2,xy^2,y^3)$ 
as a consequence of Corollary~\ref{cor:tensorprod} 
and Definition~\ref{def:vector-space-quadratic}, 
we obtain a full embedding of categories 
$\rep(\fk[x,y]/(x^2,xy^2,y^3)) \rightarrowtail \rep(\A{n}{m})$ showing that 
the representation theory of the tableaux algebra is at least as badly 
behaved as the representation theory of a finite dimensional algebra of 
wild representation type. 

An alternative reasoning without appealing to Drozd's dichotomy theorem 
and avoiding finite dimensional algebras altogether 
would be to use the also well know fact that 
there is a full embedding of categories 
$\rep(\fk\langle x_1,\dots,x_n\rangle) \rightarrowtail \rep(\fk[x_1,x_2])$ 
for all positive integers $n$, see~\cite{gelfand-ponomarev:wild}. 
Again, there is a surjective map of algebras 
$\A{n}{m} \twoheadrightarrow \fk[x,y]$ 
by Corollary~\ref{cor:tensorprod} 
and Definition~\ref{def:vector-space-quadratic}, 
yielding a full embedding of categories 
$\rep(\fk[x,y]) \rightarrowtail \rep(\A{n}{m})$. 
Thus the representation theory of the tableaux algebra is as bad as 
the representation theory of the free associative algebra. 

Furthermore, the above behavior is independent of the geometric 
environment where $\A{n}{m}$ resides, since both of the embeddings 
of the partial flag variety into projective space realizing said 
geometry have at least two degrees of freedom, as we note in 
Remark~\ref{rema:embedding-flag-variety}. 
Although this makes classifying 
the indecomposable representations of $\A{n}{m}$ a hopeless task, 
since it is commutative, 
we can give an explicit combinatorial description of its 
irreducible representations. 
 
\begin{corollary}\label{coro:finite-dimensional-irreducible}
    For any $m\geq n \in \mathbb{N}$, a 
    representation $V$ of $\A{n}{m}$ is irreducible if and only if $V$ is 
    one-dimensional. Moreover, up to isomorphism, these are classified by 
    a choice of scalars $\{ \lambda_T \in \fk \mid T \in \SSYT_m^n(1^k) 
    \text{ with } 1\leq k \leq n \}$ satisfying the relations given by $\R{n}{m}$. 
    These uniquely 
    determine the linear maps $\rho: \A{n}{m} \to \End(V)$ which induce the 
    structure maps $\rho(T) : V \to V$ given by linearly extending 
    $\rho(T)(v) = \lambda_T v$ for each 
    $T \in  \SSYT_m^n(1^k)$ with $1 \leq k \leq n$. 
\end{corollary}

As expected, these coincide with the maximal ideals of 
Section~\ref{sec:ideals-spectra}. Consequently, this gives another 
perspective of why 
$\A{n}{m}$ does not have an infinite dimensional irreducible module. 
Unfortunately, as a consequence of 
Corollary~\ref{coro:finite-dimensional-irreducible}, 
$\A{n}{m}$ cannot have a finite dimensional faithful irreducible module. 
This provides an alternative reasoning of why $\A{n}{m}$ is not primitive. 
However, since $\A{n}{m}$ is semiprimitive, 
it has a faithful semisimple module. 
We now search for it. 

The following corollary follows directly from Theorem \ref{thm:Polynomial-Iso}.
\begin{corollary}
    The algebra morphism $\Omega$ induces a faithful action 
    of $\A{n}{m}$ on $\fk[X]$. 
\end{corollary}

Unfortunately $\fk[X]$ is not semisimple as a module over $\A{n}{m}$, 
so it is not the module witnessing that $\A{n}{m}$ is semiprimitive. 
Instead, we use the maximal spectrum $\maxSpec(\A{n}{m})$. 

\begin{proposition}\label{prop:faithfulssmodule}
    The module 
    \[ 
        \bigoplus_{\underline{t} \in \mathcal{V}(\R{n}{m}) \times 
        \fk^{|\F{n}{m}|}} {\A{n}{m} / M_{\underline{t}}}
    \]
    is faithful and semisimple over $\A{n}{m}$. 
\end{proposition}

Morally, this means that to understand $\A{n}{m}$ we only need its maximal 
spectrum. 

%%%%%%%%%%%%%%%%%%%%%%%%%%%%%%%%%%%%%%%%%%%%%%%%%%%%%%%%%%%%%%%%%%%%%%%%%%%%%%%%

\section{Toric degenerations of partial flag varieties}\label{sec:toric}

In this section we prove that the varieties $\mathcal{V}(\A{n}{m})$ arise as 
toric degenerations of certain partial flag varieties. 
We recall the necessary 
constructions, but also refer the reader to~\cite{GL96} 
and~\cite[Chapter 14]{MSbook} 
for additional details and background on 
toric degenerations of flag varieties. 
In particular, given a $1$-parameter flat family, 
we say that the special fiber is a 
\newword{flat degeneration} of the generic fiber. 

We begin by observing that our varieties are indeed toric.

\begin{theorem}
    For any positive integers $m\geq n$, the variety $\mathcal{V}(\R{n}{m})$ is toric. 
\end{theorem}

\begin{proof}
    The ideal $\R{n}{m}$ is prime by Corollary~\ref{cor:tensorprod} and 
    binomial by the proof of Theorem~\ref{thm:quadratic}. 
    A prime binomial ideal defines a toric variety by~\cite{ES-binomial}. 
\end{proof}

It is a classical result that the set of semistandard Young tableaux $\SSYT_n^n$ 
is in bijection with the set of Gelfand--Tsetlin patterns $\GT_n$, 
see for example~\cite{GTpatterns}. 
As semigroups, the star operation on tableaux coincides with the usual 
additive structure on $\GT$-patterns. Consequently, when $m=n$, 
the tableaux algebra is isomorphic to the \newword{Gelfand--Tsetlin semigroup ring}: 
\[
    \A{n}{n} \cong \GGTT_n. 
\]
The \newword{Pl{\"u}cker algebra} is the quotient of the polynomial ring 
$\fk[p_\sigma]$ with $\sigma \subseteq \{1,\dots, n\}$ by the so-called 
\newword{Pl{\"u}cker relations}, see~\cite[Definition 14.5]{MSbook}. 
The set of generators $p_\sigma$, called \newword{Pl{\"u}cker coordinates}, 
are clearly in bijection with $\G{n}{n}$, 
the set of column semistandard Young tableaux. 
It was shown in~\cite{KM03} and~\cite[Corollary 14.24]{MSbook} that the 
ring of Pl{\"u}cker coordinates quotiented by certain 
initial ideal of the ideal of 
Pl{\"u}cker relations coincides with 
the Gelfand--Tsetlin semigroup ring $\GGTT_n$. 
Consider the special linear group $SL_n$ with a Borel subgroup $B$. 

\begin{theorem}[\cite{GL96, KM03}]\label{thm:degeneration-n=m}
    The Tableaux algebra $\A{n}{n}$ is a flat degeneration of the Pl{\"u}cker algebra, 
    so $\mathcal{V}(\R{n}{n})$ is a toric degeneration of the complete flag variety $SL_n/B$. 
\end{theorem}

We now generalize Theorem~\ref{thm:degeneration-n=m} and realize 
$\mathcal{V}(\R{n}{m})$ as a toric degeneration for all $m>n$. 
Fix $k < m$ and 
recall the following setup from~\cite{GL96}. Denote by $P_k \subset SL_m$ 
the maximal parabolic subgroup of upper block triangular matrices $(M_{i,j})$ 
satisfying $M_{i,j}=0$ for all $k<i\leq m$ and $1\leq j \leq k$, 
and let $W^k$ denote the set of minimal left coset generators of 
$\mathfrak{S}_m/(\mathfrak{S}_k \times \mathfrak{S}_{m-k})$, 
for $\mathfrak{S}_m$ the symmetric group on $m$ elements. 
It is easy to see that $W^k$ is in bijection with 
the set $\SSYT_m(1^k)$. Set 
\[
    \Q{n}{m} \coloneqq \bigcap_{k=1}^n P_k 
    \qquad \text{and} \qquad 
    \HH{n}{m} \coloneqq \bigcup_{k=1}^n W^k 
\]
so that $SL_m/\Q{n}{m}$ is the partial flag variety 
$\{ 0 \subsetneq V_1 \subsetneq \cdots \subsetneq V_n \subsetneq \fk^m \mid \dim(V_i)=i\}$. 
Let 
$L^{\underline{a}} \coloneqq L_{1}^{a_1}\otimes \dots \otimes L_{n}^{a_n}$ 
where $L_k$ is the ample generator of 
the Picard group $\mathrm{Pic}(SL_m/P_k)$. 
Then $\HH{n}{m}$ is an indexing set for the generators of the $\fk$-algebra 
of global sections 
$\bigoplus_{\underline{a}}\Gamma(SL_m/\Q{n}{m},L^{\underline{a}})$. 

Note that $\HH{n}{m}$ is in bijection with 
the set of columns in $\SSYT_m^n$, so $\HH{n}{m}=\G{n}{m}$ is a distributive lattice with 
partial order $\leq$ given as follows. 
Suppose $T \in \SSYT^n_m(1^a)$ and $T' \in \SSYT^n_m(1^b)$ with column reading 
words $w_a\dots w_1$ and $u_b\dots u_1$. We say $T\leq T'$ if $a\geq b$ and 
$w_i\leq u_i$ for all $1\leq i \leq b$. The join $T \vee T'$ is the column 
with $\min(a,b)$ rows with $i^{th}$ entry $\max(w_i,u_i)$. 
The meet $T \wedge T'$ is the column with $\max(a,b)$ rows with $i^{th}$ 
entry $\max(w_i,u_i)$ for $1 \leq i \leq \min(a,b)$, $i^{th}$ 
entry $w_i$ when $a > b$ and $\min(a,b) \leq i \leq \max(a,b)$, and $i^{th}$ 
entry $u_i$ when $a > b$ and $\min(a,b) \leq i \leq \max(a,b)$. 
In other words, recalling the construction in Section \ref{subsec:enumeration}, 
we have that $T \vee T' = R_{T,T'}$ and $T \wedge T'=L_{T,T'}$. 

\begin{remark}
    The partial order on $\SSYT_m(1^k)$ described above 
    is nothing more than the reverse of the natural ranked poset structure 
    coming from the crystal graph of the associated $\mathfrak{sl}_m$-representation $V(1^k)$. 
\end{remark}

\begin{lemma}\label{lem:initial-ideal}
    Let $\II{n}{m}$ be the following binomial ideal of $\fk[\HH{n}{m}]$, 
    \[
        \II{n}{m} \coloneqq \langle xy - (x\wedge y)(x\vee y)\mid 
        x,y \in \HH{n}{m} \text{ are incomparable}\rangle. 
    \]
    Then $\fk[\HH{n}{m}] = \fk[\G{n}{m}]$ and $\II{n}{m} = \R{n}{m}$. 
\end{lemma}

\begin{proof}
    Since $\HH{n}{m}= \G{n}{m}$ we need only address the second equality. 
    Suppose $xy - (x \wedge y)(x \vee y) \in \II{n}{m}$ with $x$ and $y$ 
    incomparable elements in $\G{n}{m}$ and column reading words 
    $x_k,\dots,x_1$ and $y_\ell,\dots,y_1$, respectively. 
    Without loss of generality, assume $k \geq \ell$. Then $u = x \wedge y$ 
    is the column tableau with $1\leq i \leq \ell$ entries $\min(x_i,y_i)$ 
    and $\ell<i\leq k$ entries $x_i$, and similarly 
    $v=x \vee y$ is the column tableau with $1\leq i \leq \ell$ entries 
    $\max(x_i,y_i)$. Clearly 
    $xy-uv \in \R{n}{m}$, so $\II{n}{m} \subseteq \R{n}{m}$. 

    Suppose $xy-uv \in \R{n}{m}$ for some $u,v,x,y \in \G{n}{m}$. 
    It is easy to check that $x \vee y \vee u \vee v = x \vee y = u \vee v$ and 
    $x \wedge y \wedge u \wedge v = x \wedge y = u \wedge v$, so 
    \[
        xy-uv 
        = \big[xy - (x \wedge y)(x \vee y) \big] 
        + \big[(u \wedge v)(u \vee v) - uv\big].
    \]
    If $x$ and $y$ are incomparable, and so are $u$ and $v$, then 
    $xy - (x \wedge y)(x \vee y)$ and $(u \wedge v)(u \vee v) - uv$ 
    are both in $\II{n}{m}$, so $xy-uv \in \II{n}{m}$. 
    If $x$ and $y$ are incomparable, and $u$ and $v$ are comparable, we can 
    assume without loss of generality that $u \geq v$. Then 
    $x\wedge y = u \wedge v = v$ and $x \vee y = u \vee v = u$, so 
    $xy-uv = xy - (x \vee y)(x \wedge y) \in \II{n}{m}$. 
    If $x$ and $y$ are comparable, and so are $u$ and $v$, 
    assume without loss of generality that $x \geq y$ and $u \geq v$. 
    Then $y = x \wedge y = u \wedge v = v$ and 
    $x = x \vee y = u \vee v = u$, so $xy-uv = 0 \in \II{n}{m}$. 
\end{proof}

A construction for an algebra where the generators are defined via 
the incomparable pairs in an arbitrary distributive lattice, 
like that in Lemma \ref{lem:initial-ideal}, was considered by 
Hibi in \cite{Hibi}. Thus, Lemma \ref{lem:initial-ideal} proves that 
$\A{n}{m}$ coincides with the so-called Hibi algebra of $\G{n}{m}$, 
with poset structure given by $\leq$ above. More importantly, 
Gonciulea--Lakshmibai used a similar construction in \cite{GL96} to prove 
that the algebra of global sections 
$\bigoplus_{\underline{a}}\Gamma(SL_m/Q,L^{\underline{a}})$ of any partial 
flag variety $SL_m/Q$ admits a flat degeneration given by the Hibi algebra 
of the poset of minimal coset representatives. Consequently, 
Lemma \ref{lem:initial-ideal} implies that $\A{n}{m}$ recovers 
the construction of Gonciulea--Lakshmibai for $SL_m/\Q{n}{m}$, 
and hence $\R{n}{m}$ coincides with the initial ideal of the ideal of 
Pl{\"u}cker relations for $SL_m/\Q{n}{m}$. An analogous construction 
using $\GT$-patterns was also considered in \cite{KimProtsak}. 

\begin{theorem}\label{thm:flatdeg}
    For $m>n$, the tableaux algebra $\A{n}{m}$ is a flat degeneration 
    of $\bigoplus_{\underline{a}}\Gamma(SL_m/\Q{n}{m},L^{\underline{a}})$. 
\end{theorem}

\begin{proof}
    Denote by $\{p_{\underline{\alpha}} \mid {\underline{\alpha}} \in \HH{n}{m}\}$ the generating set of 
    $\bigoplus_{\underline{a}}\Gamma(SL_m/\Q{n}{m},L^{\underline{a}})$. 
    Since $\HH{n}{m}=\G{n}{m}$ is a finite distributive lattice, 
    there exists a flat degeneration from 
    $\bigoplus_{\underline{a}}\Gamma(SL_m/\Q{n}{m},L^{\underline{a}})$ to 
    $\fk[\HH{n}{m}]/\II{n}{m}$ by~\cite[Theorem 5.2]{GL96}. 
    Since $\fk[\HH{n}{m}]/\II{n}{m} = \A{n}{m}$ 
    by Lemma~\ref{lem:initial-ideal}, the result follows. 
\end{proof}

In addition, we obtain the following. 

\begin{corollary}\label{cor:toricdegeneration}
    For $m>n$, the variety $\mathcal{V}(\R{m}{n})$ is a toric degeneration 
    of the partial flag variety $SL_m/\Q{n}{m}$. 
    In particular, the set $\BB{n}{m}$ 
    is a Gr{\"o}bner basis for $\R{m}{n}$. 
\end{corollary}

\begin{proof}
    The first statement and the fact that the set formed by the elements 
    $xy - (x\wedge y)(x\vee y)$ with $x,y \in \HH{n}{m}$ incomparable is a 
    Gr{\"o}bner basis for $\R{m}{n}$ follows from Lemma~\ref{lem:initial-ideal} 
    and~\cite[Theorem 10.6]{GL96}. 
    Given $T,T' \in \G{n}{m} = \HH{n}{m}$ we have $T \vee T' = R_{T,T'}$ and 
    $T \wedge T' = L_{T,T'}$, and it is easy to check that 
    $T$ and $T'$ are incomparable if and only if 
    $T T' - L_{T,T'} R_{T,T'} \in \BB{n}{m}$, 
    as desired. 
\end{proof}

\begin{remark}\label{rema:embedding-flag-variety}
    Note that for the complete flag variety we actually obtain two distinct flat 
    degenerations. 
    Algebraically, this distinction is readily apparent, since $\A{m}{m}$ has 
    one more free variable than $\A{m-1}{m}$. 
    Geometrically, this distinction is not so obvious. The toric degenerations 
    $\mathcal{V}(\R{m}{m})$ and $\mathcal{V}(\R{m-1}{m})$ of $SL_m/B$ obtained 
    in Theorem~\ref{thm:degeneration-n=m} and Corollary~\ref{cor:toricdegeneration}, 
    respectively, seem to be identical because of Lemma~\ref{lem:prod-ideal}. 
    The distinction stems from the fact that 
    we are using different embeddings of the 
    complete flag variety into a product of projective spaces. 
    For visual convenience and only within this remark, 
    we denote the partial flag variety by $\mathcal{F}\ell(n,m) \coloneqq 
    \{0\subsetneq V_1\subsetneq \dots \subsetneq V_{n} \subseteq \fk^m \mid \dim(V_i)=i\}$. 
    For $n = m$ we are going through the following composition of embeddings 
    \begin{equation*}
        \begin{tikzcd}[row sep=0]
            \mathcal{F}\ell(m,m) \ar[r, hook] 
            & \prod\limits_{i = 1}^{m} Gr(i,m) \ar[r, hook, "\prod \iota_{i}"] 
            & \prod\limits_{i = 1}^{m} \mathbb{P}_{\fk}^{\binom{m}{i}-1}\\
            (V_1\subsetneq \dots \subsetneq V_{m}) \ar[r, mapsto] 
            & (V_i)_{i = 1}^m 
            & {} 
        \end{tikzcd}
    \end{equation*} 
    whereas for $n = m-1$ we are going through the following composition of embeddings 
    \begin{equation*}
        \begin{tikzcd}[row sep=0]
            \mathcal{F}\ell(m-1,m) \ar[r, hook] 
            & \prod\limits_{i = 1}^{m-1} Gr(i,m) \ar[r, hook, "\prod \iota_{i}"] 
            & \prod\limits_{i = 1}^{m-1} \mathbb{P}_{\fk}^{\binom{m}{i}-1}\\
            (V_1\subsetneq \dots \subsetneq V_{m-1} \subsetneq \fk^m) \ar[r, mapsto] 
            & (V_i)_{i = 1}^{m-1} 
            & {} 
        \end{tikzcd}
    \end{equation*} 
    where $\iota_{i} \colon Gr(i,m) \to \mathbb{P}_{\fk}^{\binom{m}{i}-1}$ denotes the 
    Pl{\"u}cker embedding. The additional free variable of $\A{m}{m}$ that does not 
    appear on $\A{m-1}{m}$ neatly corresponds to the additional point 
    $\mathbb{P}_{\fk}^{\binom{m}{m}-1} = \mathbb{P}_{\fk}^{0}$ 
    that appears in the embedding for $n = m$ but not in the embedding for $n = m-1$. 
    The underlying cause of this phenomenon is the bijection between 
    a flag $0\subsetneq V_1\subsetneq \dots \subsetneq V_{m} = \fk^m$ in 
    $\mathcal{F}\ell(m,m)$ and a flag 
    $0\subsetneq V_1\subsetneq \dots \subsetneq V_{m-1} \subsetneq \fk^m$ in 
    $\mathcal{F}\ell(m-1,m)$, which induces a homeomorphism 
    $\mathcal{F}\ell(m-1,m) \cong \mathcal{F}\ell(m,m)$. 
\end{remark}

Traditionally, enumerating the minimal Pl{\"u}cker relations 
for the partial flag $SL_m/Q$ variety is a complicated computation 
that is often subdivided by separately enumerating the so-called 
\newword{Pl{\"u}cker--Grassmann relations}, encoding the Pl{\"u}cker relations 
within each Grassmann subvariety $Gr(i,m)$ of $i$-dimensional planes 
inside $\fk^m$ for all $1\leq i <m$, 
and then enumerating the \newword{incidence relations}, which encode the 
relations between Grassmannians for distinct $i$. 
As a consequence of the proof of Theorem \ref{thm:enumeration} and 
Corollary \ref{cor:toricdegeneration} we obtain the following enumerations. 

\begin{corollary}\label{cor:pluckerenumeration}
    The minimal number of Pl{\"u}cker relations in the 
    partial flag variety $SL_m/\Q{n}{m}$ is $\varsigma(\R{n}{m})$. 
    Hence, the minimal number of Pl{\"u}cker--Grassmann relations in $Gr(i,m)$ is 
    \[  \sum_{k= \max \{ 0, 2i-m \}}^{i-1} \left( \frac{1}{2} \right) \frac{i-k-1}{i-k+1} 
        \binom{m}{k} \binom{m-k}{2i-2k} \binom{2i-2k}{i-k}\]
    for each $1\leq i <m$, and the minimal number of incidence relations in $SL_m/\Q{n}{m}$ is 
    \[
         \sum_{i=1}^{n}  \sum_{k= \max \{ 0, 2i-m \}}^{i-1} 
         \sum_{j=1}^{\min\{n-i,m-2i+k\}} \frac{i-k}{i+j-k+1} \binom{m}{k} 
        \binom{m-k}{2i+j-2k}\binom{2i+j-2k}{i-k}. 
    \]
\end{corollary}

It is a classical result that $\A{n}{n} \cong \GGTT_n$ is 
Cohen--Macaulay, see~\cite[Corollary 14.25]{MSbook}. A similar proof shows that 
the same result holds for any $m \geq n$. 

\begin{theorem}\label{thm:CohenMacaulay}
    For any $m\geq n$, the algebra $\A{n}{m}$ is Cohen--Macaulay.
\end{theorem}

\begin{proof}
    Note that $\A{n}{m} / M_{\underline{t}} = \fk$ is algebraically closed 
    for every $M_{\underline{t}} \in \maxSpec(\A{n}{m})$, so $\A{n}{m}$ is 
    normal by~\cite[Lemma 030B]{stacks-project}, so $\A{n}{m}$ is an 
    integrally closed domain by 
    Theorem~\ref{prop:IntegralDomain-Noetherian-Reduced-Jacobson}. 
    Thus the semigroup $\SSYT_m^n$ is normal 
    by~\cite[Proposition 13.5]{Stu95} and affine 
    by~\cite[Theorem 7.4]{MSbook}. 
    Being the monoid algebra of a normal affine semigroup, $\A{n}{m}$ is 
    Cohen--Macaulay by~\cite[Theorem 1]{Hochster72}. 
\end{proof}

\begin{remark}\label{rema:koszul-cm-knowledge}
    As in Remark~\ref{rema:koszul-combinatorial}, 
    our proofs have the advantage of being combinatorial 
    and elementary in nature, but there are many other 
    avenues to obtain some of the structural results of 
    $\A{n}{m}$ we present. 
    For example, once we know the tableaux algebra $\A{n}{m}$ is a flat degeneration 
    and that the minimal generating set $\BB{n}{m}$ of $\R{n}{m}$ is a 
    Gr{\"o}bner basis, we can swiftly recover that $\A{n}{m}$ is Koszul 
    from the well known notion of $G$-quadratic, see \cite{conca2000G-quadratic}. 
    Similarly, once we know $\A{n}{m}$ is an integral domain and the 
    Hibi algebra of a finite poset, we obtain Cohen--Macaulayness 
    from \cite{Hibi}. 
\end{remark}

In light of Theorems~\ref{thm:flatdeg} and~\ref{thm:CohenMacaulay}, 
the following is immediate from~\cite[Corollary 8.31]{MSbook}. 

\begin{corollary}
    The ring of global sections of the partial flag variety 
    $\bigoplus_{\underline{a}}\Gamma(SL_m/\Q{n}{m},L^{\underline{a}})$ 
    is Cohen--Macaulay. 
\end{corollary}

We now showcase once again the curious behavior of the tableau algebra. 
A \newword{GCD semigroup} is a 
semigroup $S$ with the property that for any $a,b \in S$ there exists a 
$c \in S$ such that $(a+S) \cap (b+S) = c+S$. 
For $n \geq 2$, $m \geq 3$, and $m \geq n$, 
the monoid $\SSYT_m^n$ is not a GCD semigroup because 
\[
    \hackcenter{\tableau{1 & 2\\3}} 
    \in \left( \; \hackcenter{\tableau{1}} \star \SSYT_m^n \right) 
    \cap \left( \; \hackcenter{\tableau{1\\3}} \star \SSYT_m^n \right) 
    \ni \hackcenter{\tableau{1 & 1\\3}} 
\]
hence there is no column $T$ such that $T \star \SSYT_m^n$ equals the 
intersection above. 
Recall that an integral domain is a \newword{GCD domain} if any two 
elements have a greatest common divisor. In particular, a GCD domain is 
an integrally closed domain. Thus, $\A{n}{m}$ is not a GCD domain 
by~\cite[Theorem 6.4]{gilmer-parker74}, but it is an integrally closed 
domain by the proof of Theorem~\ref{thm:CohenMacaulay}. 

\begin{proposition}
    The Krull dimension of $\A{n}{m}$ is 
    \begin{align*}
        \sum_{i=1}^n {{\binom{m}{i}}} 
        \Bigg[ 
        1 + 
        \sum_{k = 1}^{\min\{i,m-i\}} 
        \binom{i}{k} 
        \Bigg[ 
        \frac{1}{2} \binom{m-i}{k} 
        - \sum_{j=k}^{\min\{n+k-i,m-i\}} \frac{k}{j+1} 
        \binom{m-i}{j} 
        \Bigg] 
        \Bigg]. 
    \end{align*}
\end{proposition}

\begin{proof}
    This follows from \cite[Proposition 7.5]{MSbook} and 
    Corollary~\ref{cor:enumeration}.
\end{proof}

%%%%%%%%%%%%%%%%%%%%%%%%%%%%%%%%%%%%%%%%%%%%%%%%%%%%%%%%%%%%%%%%%%%%%%%%%%%%%%%%

\section{Connections to string cones and crystal graphs}\label{sec:crystals}

In this section we highlight connections of the tableaux algebra 
to crystal graphs for quantum group representations, including an application to 
crystal embeddings. 

Given $\g$ a complex semisimple Lie algebra and $V(\lambda)$ 
a finite dimensional irreducible $\g$-module with highest weight $\lambda$, 
the \newword{crystal graph} of $V(\lambda)$ 
is a finite directed colored graph $\B$ 
whose vertex set is given by its \newword{crystal basis} $\B(\lambda)$. 
The edges of the crystal graph are the \newword{crystal operators} 
$e_i, f_i: \B(\lambda) \to \B(\lambda) \cup \{0\}$, which satisfy 
$e_i(b)=b'$ if and only if $f_i(b')=b$. 
Crystal graphs were introduced independently by Kashiwara and 
Lusztig~\cite{kashiwara1990crystalizing,kashiwara1991crystal,Lusztig1} 
in their study of bases for quantized universal enveloping algebras, 
and have since been instrumental 
in many groundbreaking advances in representation theory and combinatorics. 
We refer the reader to~\cite{bumpschilling:crystal-bases} 
for the details.

When $\mathfrak{g} = \mathfrak{sl}_n$ and 
$\lambda$ is a partition of at most $n$ parts, 
the crystal basis of the irreducible representation $V(\lambda)$ 
is in bijection with the set $\SSYT_n(\lambda)$. Hence, we identify the $\mathfrak{sl}_n$-crystal 
$\B(\lambda)$ with $\SSYT_n(\lambda)$ endowed with 
the action of certain crystal operators $e_i$ and $f_i$ with 
$1\leq i <n$ (see \cite{kashiwara1994crystal,bumpschilling:crystal-bases} for 
a precise combinatorial description). 
In this particular case, the \newword{weight map}, with which any crystal 
graph comes equipped, is precisely the map 
$\wt : \SSYT_n(\lambda) \to \mathbb{Z}_{\geq 0}^{n}$ 
induced by Definition~\ref{def:rowweight}, namely 
$\wt(T) = \sum_{i=1}^n \wt^{(i)}(T)$ for $T \in \SSYT_n(\lambda)$. 

\begin{definition}
    Let $\cE$ and $\cF$ be the free monoids generated by 
    $\{e_i\}_{i=1}^{n-1}$ and $\{f_i\}_{i=1}^{n-1}$, respectively, 
    which act on $\B$ via concatenation of operators. 
    The unique element $b \in \B(\lambda)$ satisfying 
    $\cE\{b\} = \{b\}$ or $\cF\{b\} = \{b\}$ is called the 
    \newword{highest weight vector} or the 
    \newword{lowest weight vector}, respectively. 
\end{definition}

The highest and lowest weight vectors of $\B(\lambda) = \SSYT_n(\lambda)$ 
are fully characterized by their fillings. 
The highest weight vector $T \in \SSYT_n(\lambda)$ satisfies that 
all cells in the $i^{th}$ row of $T$ have entry $i$, and 
the lowest weight vector $T \in \SSYT_n(\lambda)$ satisfies that 
the $j^{th}$ column of $T$ has entries $\{n-h_j+1, \dots, n\}$, 
where $h_j$ is the height of the column. 

\begin{definition}
    A map $\Phi: \mathscr{B} \rightarrow \mathscr{C}$ between two $\mathfrak{g}$-crystals 
    is an \newword{embedding} when it is injective, $e_{i}(x) \neq 0$ 
    implies $\Phi(e_{i}(x))=e_{i}(\Phi(x)) \neq 0$, and $f_{i}(x) \neq 0$ 
    implies $\Phi(f_{i}(x))=f_{i}(\Phi(x)) \neq 0$. 
\end{definition}

\begin{remark}
    Note that this condition differs with the conventional definition 
    of a crystal morphism, where an ``if and only if'' condition is employed. 
\end{remark}

Introduced by 
Littelmann~\cite{conescrystalspatterns} and 
Berenstein--Zelevinsky~\cite{BerensteinZelevinsky}, 
string cones and string polytopes' original purpose was to 
parametrize the elements in Lusztig's dual canonical 
basis~\cite{Lusztig1, Lusztig2}. These objects have important connections 
to toric degenerations of Schubert and flag varieties~\cite{caldero:toric}, 
and cluster algebras~\cite{bea_polyhedral, BossingerFourier}. 

Let $w_0$ be the longest word in the symmetric group $\mathfrak{S}_n$. 
For any given reduced word $u$ of $w_0$ and $\lambda$ a partition with 
at most $n$ rows, denote by $S_u$ the \newword{string cone} associated 
to $u$ and by $S_u(\lambda)$ the \newword{string polytope} 
associated to $\lambda$ with respect to the decomposition $u$. 
In particular, $S_u$ is an abelian semigroup under 
addition~\cite{BerensteinZelevinsky} and $S_u(\lambda)$ carries an 
$\mathfrak{sl}_n$-crystal structure~\cite{GENZ}. 

Fix $u = (1,2,1,3,2,1, \ldots, n-1, n-2, \ldots, 2,1)$ 
for the rest of the section and let 
$\underline{b} \in S_u(\mu)$; In general, there is an injective map 
\begin{equation}
    \label{eq:stringmap}
    \begin{tikzcd}[row sep=0]
        \Phi_{\underline{b}} \colon S_{u}(\lambda) \ar[r] 
        & S_{u}(\lambda + \mu)\\
        \phantom{\Phi_{\underline{b}} \colon} \underline{a} \ar[r, mapsto] 
        & \underline{a} + \underline{b}
    \end{tikzcd}
\end{equation}
which in~\cite[Corollary 4.5]{BT25} 
was shown to be a weight-preserving crystal embedding for
$\mu = (2, 1^{n-2})$ the highest root and 
$\underline{b}$ any weight zero vector in 
the adjoint representation $V(2, 1^{n-2})$. 

There is a well known bijection 
$\mathsf{S_\lambda} : S_u(\lambda) \to \SSYT_n(\lambda)$, 
between 
string cones and the set of semistandard Young tableaux 
that factors through Gelfand--Tsetlin patterns and 
induces a bijection on the corresponding crystal 
graphs~\cite[Corollary 2]{conescrystalspatterns}. 
It is straightforward to see that it also intertwines the monoid 
structures, namely 
$\mathsf{S}_{\lambda+\mu}(\underline{a} + \underline{b}) 
= \mathsf{S}_\lambda(\underline{a}) \star \mathsf{S}_\mu(\underline{b})$, 
see \cite[Section 2.3]{alexandersson:gt}. 
Thus for any fixed $T \in \SSYT_n(\mu)$, the map~\eqref{eq:stringmap} 
induces the following well-defined injective map on crystals of 
semistandard Young tableaux. 
\begin{equation*}
    \begin{tikzcd}[row sep=0]
        \Phi_T \colon \B(\lambda) \ar[r] 
        & \B(\lambda + \mu)\\
        \phantom{\Phi_T \colon} T' \ar[r, mapsto] 
        & T' \star T
    \end{tikzcd}
\end{equation*}

For $T \in \SSYT_n$, recall the \newword{column reading word} $w_{col}(T)$ 
from Definition~\ref{def:colreadingword}, and define analogously the 
\newword{row reading word} $w_{row}(T)$, obtained by 
reading the entries of $T$ from right to left in the rows 
starting with the topmost row and moving down. 
In what follows, for any partition $\lambda$ denote by 
$A_\lambda$ and $Z_\lambda$ the highest and lowest 
weight vectors of $\B(\lambda)$, respectively. 

\begin{lemma}\label{lem:insertion}
    Given any $T \in \SSYT_n(\lambda)$, the row insertion of $w_{col}(Z_\mu)$ 
    into $T$ coincides with $T \star Z_\mu$ and the column insertion of 
    $w_{row}(A_\mu)$ into $T$ coincides with $T \star A_\mu$. 
\end{lemma}

\begin{proof}
    This follows directly from RSK insertion 
    and the characterizations given above. 
    Namely, $w_{col}(Z_\mu)$ starts by inserting $m$ into $T$, which will 
    always be placed in the first row. Then we insert $m-1$, which will 
    be placed in the first row and bump the previous $m$ to the second row. 
    Continuing this process, 
    once we have inserted the first column of $w_{col}(Z_\mu)$ into $T$, 
    we will have exactly the star product of $T$ with 
    the first column of $w_{col}(Z_\mu)$. Induction on the columns of 
    $w_{col}(Z_\mu)$ yields 
    $[T \xleftarrow{\text{row}} w_{col}(Z_\mu)] = T \star Z_\mu$, 
    and a dual argument exchanging rows and columns yields 
    $[T \xleftarrow{\text{col}} w_{row}(A_\mu)] = T \star A_\mu$. 
\end{proof}

\begin{remark}
    It is important to note that although row and column insertion happen to coincide 
    with the star product in the special case of highest and lowest weights, 
    this is not true in general. For instance, consider the tableaux $T$ of shape 
    $(2,1)$ with $w_{col}(T)=312$. Then, then $T\star T$ has partition shape 
    $(4,2)$ but $[T \xleftarrow{\text{row}} w_{col}(T)]$ has partition shape $(3,2,1)$. 
\end{remark}

We now extend~\cite[Corollary 4.5]{BT25} to include 
highest and lowest weight vectors. 

\begin{theorem}\label{thm:crystal}
    Let $T$ be the highest or lowest weight vector of $\B(\mu)$. 
    Then $\Phi_T: \B(\lambda) \to \B(\lambda+ \mu)$ 
    is a crystal embedding. 
    In particular, the images are the induced subgraphs 
    $\Phi_{A_\mu}(\B(\lambda)) = \cE(Z_\lambda \star A_\mu)$ and 
    $\Phi_{Z_\mu}(B(\lambda)) = \cF(A_\lambda \star Z_\mu)$, and 
    yield isomorphic embeddings of 
    $\B(\lambda)$ into $\B(\lambda + \mu)$. 
\end{theorem}

\begin{proof}
    Suppose $T=A_\mu$, since 
    $\wt(A_\lambda \star A_\mu) = \wt(A_{\lambda+\mu})$, then 
    $\Phi_{A_\mu}(A_\lambda)=A_\lambda \star A_\mu = A_{\lambda+\mu}$ 
    by uniqueness of the highest weight vector. 
    The map $\Phi_{A_\mu}$ on $\B(\lambda)$ coincides with column insertion of 
    $A_\mu$ into all tableaux in $\SSYT_n(\lambda)$ by 
    Lemma~\ref{lem:insertion}, and since insertion is a 
    well-defined crystal morphism, then $\Phi_{A_\mu}(\B(\lambda))$ 
    is the embedded subcrystal given by 
    $\cE(\Phi_{A_\mu}(Z_\lambda)) = \cE(Z_\lambda \star A_\mu)$. 

    Similarly, if $T = Z_\mu$ then 
    $\Phi_{Z_\mu}(Z_\lambda) = Z_\lambda \star Z_\mu = Z_{\lambda+\mu}$, 
    and again by Lemma~\ref{lem:insertion} and the fact that row insertion 
    is a crystal morphism we obtain 
    $\Phi_{Z_\mu}(\B(\lambda)) 
    = \cF(\Phi_{Z_\mu}(A_\lambda)) 
    = \cF(A_\lambda \star Z_\mu)$. 
\end{proof}

In essence, the maps above pick out copies of $\B(\lambda)$ inside each 
connected component of their tensor products 
$\B(\lambda) \otimes \B(\eta) = \bigoplus_\nu \B(\nu)$. 
Namely, for each $\mu$ satisfying $\nu = \lambda + \mu$, 
the maps $\Phi_{A_\mu}$ and $\Phi_{Z_\mu}$ yield two ways of 
identifying $\B(\lambda)$ within $\B(\nu)$; 
by mapping the highest weight vector to the highest weight vector and 
by mapping the lowest weight vector to the lowest weight vector, 
respectively. 
In the particular case of $\eta=\mu$, 
the maps $\Phi_{A_\mu}$ and $\Phi_{Z_\mu}$ identify $\B(\lambda)$ 
within the leading connected component 
$\B(\lambda+\mu)$ of $\B(\lambda) \otimes \B(\mu)$. 

%%%%%%%%%%%%%%%%%%%%%%%%%%%%%%%%%%%%%%%%%%%%%%%%%%%%%%%%%%%%%%%%%%%%%%%%%%%%%%%%

\bibliographystyle{plain}
\bibliography{main.bib}

\begin{thebibliography}{10}

\bibitem{alexandersson:gt}
P.~Alexandersson.
\newblock Gelfand-{T}setlin polytopes and the integer decomposition property.
\newblock {\em European J. Combin.}, 54:1--20, 2016.

\bibitem{balmer2005quadratic}
P.~Balmer.
\newblock Products of degenerate quadratic forms.
\newblock {\em Compos. Math.}, 141(6):1374--1404, 2005.

\bibitem{mirrorsymmetry}
V.~V. Batyrev, I.~Ciocan-Fontanine, B.~Kim, and D.~van Straten.
\newblock Mirror symmetry and toric degenerations of partial flag manifolds.
\newblock {\em Acta Math.}, 184(1):1--39, 2000.

\bibitem{BerensteinZelevinsky}
A.~Berenstein and A.~Zelevinsky.
\newblock Tensor product multiplicities, canonical bases and totally positive
  varieties.
\newblock {\em Invent. Math.}, 143(1):77--128, 2001.

\bibitem{BossingerFourier}
L.~Bossinger and G.~Fourier.
\newblock String cone and superpotential combinatorics for flag and {S}chubert
  varieties in type {A}.
\newblock {\em J. Combin. Theory Ser. A}, 167:213--256, 2019.

\bibitem{BT25}
L.~Bossinger and J.~Torres.
\newblock Unveiling crystal embeddings: New perspectives on string polytopes
  and atomic decompositions.
\newblock Preprint, arXiv:2505.22127, 2025.

\bibitem{bumpschilling:crystal-bases}
D.~Bump and A.~Schilling.
\newblock {\em Crystal bases}.
\newblock World Scientific Publishing Co. Pte. Ltd., Hackensack, NJ, 2017.
\newblock Representations and combinatorics.

\bibitem{caldero:toric}
P.~Caldero.
\newblock Toric degenerations of {S}chubert varieties.
\newblock {\em Transform. Groups}, 7(1):51--60, 2002.

\bibitem{chang-duan-fraser-li:quantum}
W.~Chang, B.~Duan, C.~Fraser, and J.-R. Li.
\newblock Quantum affine algebras and {G}rassmannians.
\newblock {\em Math. Z.}, 296(3-4):1539--1583, 2020.

\bibitem{chiviri}
R.~Chiriv{\`i}.
\newblock L{S} algebras and application to {S}chubert varieties.
\newblock {\em Transform. Groups}, 5(3):245--264, 2000.

\bibitem{conca2000G-quadratic}
A.~Conca.
\newblock Gr\"obner bases for spaces of quadrics of low codimension.
\newblock {\em Adv. in Appl. Math.}, 24(2):111--124, 2000.

\bibitem{dehy}
R.~Dehy and R.~W.~T. Yu.
\newblock Degeneration of {S}chubert varieties of {${\rm SL}_n/B$} to toric
  varieties.
\newblock {\em Ann. Inst. Fourier (Grenoble)}, 51(6):1525--1538, 2001.

\bibitem{drozd:wild}
J.~A. Drozd.
\newblock Representations of commutative algebras.
\newblock {\em Funkcional. Anal. i Prilo\v zen.}, 6(4):41--43, 1972.

\bibitem{ES-binomial}
D.~Eisenbud and B.~Sturmfels.
\newblock Binomial ideals.
\newblock {\em Duke Math. J.}, 84(1):1--45, 1996.

\bibitem{fgl}
X.~Fang, G.~Fourier, and P.~Littelmann.
\newblock On toric degenerations of flag varieties.
\newblock In {\em Representation theory---current trends and perspectives}, EMS
  Ser. Congr. Rep., pages 187--232. Eur. Math. Soc., Z\"urich, 2017.

\bibitem{GTpatterns}
I.~M. Gel'fand and M.~L. Cetlin.
\newblock Finite-dimensional representations of the group of unimodular
  matrices.
\newblock {\em Doklady Akad. Nauk SSSR (N.S.)}, 71:825--828, 1950.

\bibitem{gelfand-ponomarev:wild}
I.~M. Gel'fand and V.~A. Ponomarev.
\newblock Remarks on the classification of a pair of commuting linear
  transformations in a finite-dimensional space.
\newblock {\em Funkcional. Anal. i Prilo\v zen.}, 3(4):81--82, 1969.

\bibitem{bea_polyhedral}
V.~Genz, G.~Koshevoy, and B.~Schumann.
\newblock Polyhedral parametrizations of canonical bases \& cluster duality.
\newblock {\em Adv. Math.}, 369:107178, 41, 2020.

\bibitem{GENZ}
V.~Genz, G.~Koshevoy, and B.~Schumann.
\newblock Combinatorics of canonical bases revisited: string data in type
  {$A$}.
\newblock {\em Transform. Groups}, 27(3):867--895, 2022.

\bibitem{gilmer:semigroup}
R.~Gilmer.
\newblock {\em Commutative semigroup rings}.
\newblock Chicago Lectures in Mathematics. University of Chicago Press,
  Chicago, IL, 1984.

\bibitem{gilmer-parker74}
R.~Gilmer and T.~Parker.
\newblock Divisibility properties in semigroup rings.
\newblock {\em Michigan Math. J.}, 21:65--86, 1974.

\bibitem{GL96}
N.~Gonciulea and V.~Lakshmibai.
\newblock Degenerations of flag and {S}chubert varieties to toric varieties.
\newblock {\em Transform. Groups}, 1(3):215--248, 1996.

\bibitem{hrw1998koszul}
J.~Herzog, V.~Reiner, and V.~Welker.
\newblock The {K}oszul property in affine semigroup rings.
\newblock {\em Pacific J. Math.}, 186(1):39--65, 1998.

\bibitem{Hibi}
T.~Hibi.
\newblock Distributive lattices, affine semigroup rings and algebras with
  straightening laws.
\newblock In {\em Commutative algebra and combinatorics ({K}yoto, 1985)},
  volume~11 of {\em Adv. Stud. Pure Math.}, pages 93--109. North-Holland,
  Amsterdam, 1987.

\bibitem{Hochster72}
M.~Hochster.
\newblock Rings of invariants of tori, {C}ohen-{M}acaulay rings generated by
  monomials, and polytopes.
\newblock {\em Ann. of Math. (2)}, 96:318--337, 1972.

\bibitem{HM11}
B.~J. Howard and T.~B. McAllister.
\newblock Degree bounds for type-{A} weight rings and {G}elfand-{T}setlin
  semigroups.
\newblock {\em J. Algebraic Combin.}, 34(2):237--249, 2011.

\bibitem{kashiwara1990crystalizing}
M.~Kashiwara.
\newblock Crystalizing the {$q$}-analogue of universal enveloping algebras.
\newblock {\em Comm. Math. Phys.}, 133(2):249--260, 1990.

\bibitem{kashiwara1991crystal}
M.~Kashiwara.
\newblock On crystal bases of the {$Q$}-analogue of universal enveloping
  algebras.
\newblock {\em Duke Math. J.}, 63(2):465--516, 1991.

\bibitem{kashiwara1994crystal}
M.~Kashiwara.
\newblock Crystal bases of modified quantized enveloping algebra.
\newblock {\em Duke Math. J.}, 73(2):383--413, 1994.

\bibitem{KimProtsak}
S.~Kim and V.~Protsak.
\newblock Hibi algebras and representation theory.
\newblock {\em Acta Math. Vietnam.}, 44(1):307--323, 2019.

\bibitem{knuth}
D.~E. Knuth.
\newblock Permutations, matrices, and generalized {Y}oung tableaux.
\newblock {\em Pacific J. Math.}, 34:709--727, 1970.

\bibitem{KnutsonTao}
A.~Knutson and T.~Tao.
\newblock The honeycomb model of {${\rm GL}_n({\bf C})$} tensor products. {I}.
  {P}roof of the saturation conjecture.
\newblock {\em J. Amer. Math. Soc.}, 12(4):1055--1090, 1999.

\bibitem{KM03}
M.~Kogan and E.~Miller.
\newblock Toric degeneration of {S}chubert varieties and {G}elfand-{T}setlin
  polytopes.
\newblock {\em Adv. Math.}, 193(1):1--17, 2005.

\bibitem{superplactic}
R.~La~Scala, V.~Nardozza, and D.~Senato.
\newblock Super {RSK}-algorithms and super plactic monoid.
\newblock {\em Internat. J. Algebra Comput.}, 16(2):377--396, 2006.

\bibitem{Lakshmibai}
V.~Lakshmibai.
\newblock Degenerations of flag varieties to toric varieties.
\newblock {\em C. R. Acad. Sci. Paris S\'er. I Math.}, 321(9):1229--1234, 1995.

\bibitem{BL09}
V.~Lakshmibai and J.~Brown.
\newblock {\em Flag varieties}, volume~53 of {\em Texts and Readings in
  Mathematics}.
\newblock Hindustan Book Agency, New Delhi, 2009.
\newblock An interplay of geometry, combinatorics, and representation theory.

\bibitem{LS-plactic}
A.~Lascoux and M.-P. Sch\"utzenberger.
\newblock Le mono\"ide plaxique.
\newblock In {\em Noncommutative structures in algebra and geometric
  combinatorics ({N}aples, 1978)}, volume 109 of {\em Quad. ``Ricerca Sci.''},
  pages 129--156. CNR, Rome, 1981.

\bibitem{li:canonical}
J.-R. Li.
\newblock Dual canonical bases for unipotent groups and base affine spaces.
\newblock Preprint, arXiv:2010.07060, 2022.

\bibitem{conescrystalspatterns}
P.~Littelmann.
\newblock Cones, crystals, and patterns.
\newblock {\em Transform. Groups}, 3(2):145--179, 1998.

\bibitem{Lothaire}
M.~Lothaire.
\newblock {\em Algebraic combinatorics on words}, volume~90 of {\em
  Encyclopedia of Mathematics and its Applications}.
\newblock Cambridge University Press, Cambridge, 2002.
\newblock A collective work by J. Berstel, D. Perrin, P. Seebold, J. Cassaigne,
  A. De Luca, S. Varricchio, A. Lascoux, B. Leclerc, J.-Y. Thibon, V. Bruyere,
  C. Frougny, F. Mignosi, A. Restivo, C. Reutenauer, D. Foata, G.-N. Han, J.
  Desarmenien, V. Diekert, T. Harju, J. Karhumaki, and W. Plandowski, with a
  preface by J. Berstel and D. Perrin.

\bibitem{Lusztig1}
G.~Lusztig.
\newblock Canonical bases arising from quantized enveloping algebras.
\newblock {\em J. Amer. Math. Soc.}, 3(2):447--498, 1990.

\bibitem{Lusztig2}
G.~Lusztig.
\newblock Finite-dimensional {H}opf algebras arising from quantized universal
  enveloping algebra.
\newblock {\em J. Amer. Math. Soc.}, 3(1):257--296, 1990.

\bibitem{Mac1909}
P.~A. MacMahon.
\newblock Memoir on the theory of the partitions of numbers. part iv: On the
  probability that the successful candidate at an election by ballot may never
  at any time have fewer votes than the one who is unsuccessful; on a
  generalization of this question; and on its connexion with other questions of
  partition, permutation, and combination.
\newblock {\em Philosophical Transactions of the Royal Society of London.
  Series A, Containing Papers of a Mathematical or Physical Character},
  209:153--175, 1909.

\bibitem{MSbook}
E.~Miller and B.~Sturmfels.
\newblock {\em Combinatorial commutative algebra}, volume 227 of {\em Graduate
  Texts in Mathematics}.
\newblock Springer-Verlag, New York, 2005.

\bibitem{poirier1995algebres}
S.~Poirier and C.~Reutenauer.
\newblock Alg{\`e}bres de {H}opf de tableaux.
\newblock {\em Ann. Sci. Math. Qu{\'e}bec}, 19(1):79--90, 1995.

\bibitem{polishchuk2005}
A.~Polishchuk and L.~Positselski.
\newblock {\em Quadratic algebras}, volume~37 of {\em University Lecture
  Series}.
\newblock American Mathematical Society, Providence, RI, 2005.

\bibitem{ringel:rep-type}
C.~M. Ringel.
\newblock The representation type of local algebras.
\newblock In {\em Proceedings of the {I}nternational {C}onference on
  {R}epresentations of {A}lgebras ({C}arleton {U}niv., {O}ttawa, {O}nt.,
  1974)}, volume No. 9 of {\em Carleton Mathematical Lecture Notes}, pages
  Paper No. 22, 24. Carleton Univ., Ottawa, ON, 1974.

\bibitem{Seshadri}
C.~S. Seshadri.
\newblock {\em Introduction to the theory of standard monomials}, volume~46 of
  {\em Texts and Readings in Mathematics}.
\newblock Hindustan Book Agency, New Delhi, 2007.
\newblock With notes by Peter Littelmann and Pradeep Shukla, Appendix A by V.
  Lakshmibai, Revised reprint of lectures published in the Brandeis Lecture
  Notes series.

\bibitem{stacks-project}
The {Stacks project authors}.
\newblock The {Stacks} project.
\newblock \url{https://stacks.math.columbia.edu}, 2025.

\bibitem{Stu95}
B.~Sturmfels.
\newblock {\em Gr\"obner bases and convex polytopes}, volume~8 of {\em
  University Lecture Series}.
\newblock American Mathematical Society, Providence, RI, 1996.

\end{thebibliography}
\end{document}